\documentclass[a4paper,12pt,final]{amsart}
\usepackage{times,a4wide,mathrsfs,amssymb,amsmath,amsthm,wasysym}

\newcommand{\C}{\mathbb{C}}

\newcommand{\QQ}{\mathbb{Q}}
\newcommand{\NN}{\mathbb{N}}
\newcommand{\PP}{\mathbb{P}}
\newcommand{\LLL}{\mathbb{L}}

\newcommand{\OO}{\mathcal O}
\newcommand{\Ss}{\mathcal S}

\newcommand{\Sy}{\mathfrak S}

\newcommand{\DD}{\mathcal D}
\newcommand{\XX}{\mathcal X}
\newcommand{\YY}{\mathcal Y}

\newcommand{\AAA}{\mathcal A}

\newcommand{\Cc}{\mathcal C}
\newcommand{\Zz}{\mathcal Z}

\newcommand{\MM}{\mathcal M}

\newcommand{\gr}{\hbox{Gr}}
\newcommand{\wt}{\widetilde}
\newcommand{\rom}{\romannumeral}

 \makeatletter
\newcommand*{\da@rightarrow}{\mathchar"0\hexnumber@\symAMSa 4B }
\newcommand*{\da@leftarrow}{\mathchar"0\hexnumber@\symAMSa 4C }
\newcommand*{\xdashrightarrow}[2][]{%
  \mathrel{%
    \mathpalette{\da@xarrow{#1}{#2}{}\da@rightarrow{\,}{}}{}%
  }%
}
\newcommand{\xdashleftarrow}[2][]{%
  \mathrel{%
    \mathpalette{\da@xarrow{#1}{#2}\da@leftarrow{}{}{\,}}{}%
  }%
}
\newcommand*{\da@xarrow}[7]{%
  \sbox0{$\ifx#7\scriptstyle\scriptscriptstyle\else\scriptstyle\fi#5#1#6\m@th$}%
  \sbox2{$\ifx#7\scriptstyle\scriptscriptstyle\else\scriptstyle\fi#5#2#6\m@th$}%
  \sbox4{$#7\dabar@\m@th$}%
  \dimen@=\wd0 %
  \ifdim\wd2 >\dimen@
    \dimen@=\wd2 %
  \fi
  \count@=2 %
  \def\da@bars{\dabar@\dabar@}%
  \@whiledim\count@\wd4<\dimen@\do{%
    \advance\count@\@ne
    \expandafter\def\expandafter\da@bars\expandafter{%
      \da@bars
      \dabar@ 
    }%
  }%
  \mathrel{#3}%
  \mathrel{%
    \mathop{\da@bars}\limits
    \ifx\\#1\\%
    \else
      _{\copy0}%
    \fi
    \ifx\\#2\\%
    \else
      ^{\copy2}%
    \fi
  }%
  \mathrel{#4}%
}
\makeatother

\DeclareMathOperator{\ide}{id}

\newtheorem{theorem}{Theorem}[section]
\newtheorem{claim}[theorem]{Claim}

\newtheorem{corollary}[theorem]{Corollary}
\newtheorem{proposition}[theorem]{Proposition}
\newtheorem{conjecture}[theorem]{Conjecture}
\newtheorem{remark}[theorem]{Remark}
\newtheorem{definition}[theorem]{Definition}
\newtheorem{convention}{Conventions}

\newtheorem{notation}[theorem]{Notation}

\newtheorem{nonumbering}{Theorem}

\newtheorem{nonumberingt}{Acknowledgements}

\begin{document}
\author[Robert Laterveer]
{Robert Laterveer}

\address{Institut de Recherche Math\'ematique Avanc\'ee,
CNRS -- Universit\'e 
de Strasbourg,\
7 Rue Ren\'e Des\-car\-tes, 67084 Strasbourg CEDEX,
FRANCE.}
\email{robert.laterveer@math.unistra.fr}

\title[Lagrangian subvarieties in the Chow ring of HK varieties]{Lagrangian subvarieties in the Chow ring of some hyperk\"ahler varieties}

\begin{abstract} Let $X$ be a hyperk\"ahler variety, and let $Z\subset X$ be a Lagrangian subvariety. Conjecturally, $Z$ should have trivial intersection with certain parts of the 
Chow ring of $X$. We prove this conjecture for certain Hilbert schemes $X$ having a Lagrangian fibration, and $Z\subset X$ a general fibre of the Lagrangian fibration.
 \end{abstract}


\maketitle

\section{Introduction}

For a smooth projective variety $X$ over $\C$, let $A^i(X):=CH^i(X)_{\QQ}$ denote the Chow groups (i.e. the groups of codimension $i$ algebraic cycles on $X$ with $\QQ$--coefficients, modulo rational equivalence). 

The world of Chow groups is a huge building site that is still under construction, with many unfinished parts that only exist ``in pencil'', i.e. dependent on conjectures \cite{B}, \cite{J2}, \cite{J4}, \cite{Kim}, \cite{Mur}, \cite{Vo}, \cite{MNP}. In this building site, one place of particular interest is occupied by hyperk\"ahler varieties (i.e. projective irreducible holomorphic symplectic manifolds \cite{Beau1}, \cite{Beau0}). Here, recent years have seen an intense activity of new constructions and significant progress in the understanding of Chow groups \cite{Beau3}, \cite{V13}, \cite{V14}, \cite{V17}, \cite{SV}, \cite{V6}, \cite{Rie}, \cite{Rie2}, \cite{LFu}, \cite{LFu2}, \cite{Lin}, \cite{Lin2}, \cite{FTV}.
Much of this progress has centered around the following conjecture:

\begin{conjecture}[Beauville, Voisin \cite{Beau3}, \cite{V17}]\label{conjbv} Let $X$ be a hyperk\"ahler variety. Let $D^\ast(X)\subset A^\ast(X)$ denote the $\QQ$--subalgebra generated by divisors and Chern classes of $X$. Then the cycle class maps induce injections
  \[ D^i(X)\ \hookrightarrow\ H^{2i}(X,\QQ)\ \ \ \forall i\ .\] 
  \end{conjecture}
  
  (cf. \cite{Beau3}, \cite{V17}, \cite{BV}, \cite{Fer}, \cite{Rie2}, \cite{LFu}, \cite{Yin} for cases where conjecture \ref{conjbv} is satisfied.)
  
  The ``motivation'' underlying conjecture \ref{conjbv} is that for a hyperk\"ahler variety $X$, the Chow ring $A^\ast(X)$ is expected to have a bigrading $A^\ast_{[\ast]}(X)$, where the piece $A^i_{[j]}(X)$ corresponds to the graded $\gr^j_F A^i(X)$ for the conjectural Bloch--Beilinson filtration. In particular, it is expected that the subring $A^\ast_{[0]}(X)$ injects into cohomology, and that $D^\ast(X)\subset A^\ast_{[0]}(X)$.
  
 In addition to divisors and Chern classes, what other cycles should be in the subring $A^\ast_{[0]}(X)$ (assuming this subring exists) ?
 A conjecture of Voisin provides more candidat members:
 
 \begin{conjecture}[Voisin \cite{V14}]\label{conjv} Let $X$ be a hyperk\"ahler variety of dimension $n=2m$. Let $Z\subset X$ be a codimension $i$ subvariety swept out by $i$--dimensional constant cycle subvarieties. There exists a subring $A^\ast_{[0]}(X)\subset A^\ast(X)$ injecting into cohomology, and
   \[ Z\ \ \in A^i_{[0]}(X)\ .\]
   \end{conjecture}
   
  A {\em constant cycle subvariety\/} is by definition a closed subvariety $T \subset X$ such that the image of the natural map $A_0(T)\to A^n(X)$ has dimension $1$. In particular, conjecture \ref{conjv} stipulates that Lagrangian constant cycle subvarieties (i.e., constant cycle subvarieties of dimension $m$) should lie in $A^m_{[0]}(X)$.
  
  Another conjecture concerns the behaviour of Lagrangian subvarieties (i.e. $m$--dimensional subvarieties $Z\subset X$ such that the symplectic form of $X$ restricts to $0$ on the regular part of $Z$) with respect to the intersection product. The Lagrangian condition implies that
  \[  \cup Z\colon\ \ \ H^{2,0}(X)\ \to\ H^{m+2,m}(X) \]
  is the zero map.
    Since $H^{\ast,0}(X)$ is generated by $H^{2,0}(X)$, we have that
   \[ \cup Z\colon\ \ \ H^{j,0}(X)\ \to\ H^{m+j,m}(X) \]
  is the zero map for all $j>0$.   
  Since conjecturally, the piece $A^j_{[j]}(X)$ is determined by $H^{j,0}(X)$, and the piece $A^{m+j}_{[j]}(X)$ is determined by $H^{2m+j}(X)$, we arrive at the following conjecture:
  
  \begin{conjecture}\label{conjco} Let $X$ be a hyperk\"ahler variety of dimension $2m$. Let $Z\subset X$ be a Lagrangian subvariety.
   Then the maps
    \[ \begin{split}    &A^j_{[j]}(X) \ \xrightarrow{\cdot Z}\  A^{m+j}(X)   \ \to\ A^{m+j}_{[j]}(X) \ ,\\
                                  &A^m_{[j]}(X) \ \xrightarrow{\cdot Z}\  A^{2m}(X)   \ \to\ A^{2m}_{[j]}(X) \ ,\\
       \end{split}    \]
   are zero for all $j>0$. (Here, the right arrows are projection to the piece $A^{\ast}_{[j]}(X)$.)
     \end{conjecture}

The goal of this note is to provide some examples where conjectures \ref{conjv} 
and \ref{conjco} 
are satisfied, by looking at Hilbert schemes of $K3$ surfaces. Here, thanks to work of Vial \cite{V6} and of Shen--Vial \cite{SV}, the Chow groups of $X$ split in a finite number of pieces $A^\ast_{(\ast)}(X).$\footnote{NB: we will reserve the notation $A^\ast_{(\ast)}()$ for the bigrading that is constructed unconditionally in \cite{SV}, \cite{V6} for certain hyperk\"ahler varieties. The notation $A^\ast_{[\ast]}()$, that occurs only in this introduction, refers to a {\em conjectural\/} bigrading with the property that $A^i_{[j]}(X)$ is related to the graded $\gr^j_F A^i(X)$ for the conjectural Bloch--Beilinson filtration. In short, the (unconditionally existing) bigrading $A^\ast_{(\ast)}()$ is a candidate for the (only ideally existing) bigrading $A^\ast_{[\ast]}()$.}

 The first series of examples consists of Hilbert squares $X=S^{[2]}$, where $S$ is a general $K3$ surface of genus $g$. If the integer $g$ satisfies $2g-2=2m^2$ for some integer $m\ge 2$, the Hilbert square $X$ admits a Lagrangian fibration $\phi\colon X\to\PP^2$ \cite{HT} (cf. subsection \ref{ssfib}). The general fibre $A$ of $\phi$ is Lagrangian; it thus makes sense to ask whether conjecture \ref{conjco} is true for $A\subset X$.
 We give an answer for the first two values of $g$:
 
 \begin{nonumbering}[=theorem \ref{main4} and corollary \ref{cor4}] Let $X=S^{[2]}$, where $S$ is a general $K3$ surface of genus $g=5$ or $g=10$. Let $A\subset X$ be a general fibre of the Lagrangian
 fibration $\phi$. Then $A\in A^2_{(0)}(X)$ and
   \[ \cdot A\colon\ \ \ A^2_{hom}(X)\ \to\ A^4(X) \]
   is the zero map.
   
   (In particular, let $b\in A^4(X)$ be a $0$--cycle of the form $b=A\cdot c$ with $c\in A^2(X)$. Then $b$ is rationally trivial if and only if $b$ has degree $0$.)
  \end{nonumbering} 

(The $g=5$ case of theorem \ref{main4} was already done in \cite{GFC}.)   

The second series of examples consists of Hilbert cubes $X=S^{[3]}$, where $S$ is a general $K3$ surface of genus $g$. For $g=9$, the Hilbert cube $X$ admits a Lagrangian 
fibration $\phi\colon X\to\PP^3$ \cite{IR} (cf. subsection \ref{ssfib}). We establish a weak version of conjecture \ref{conjco} for this case:

 \begin{nonumbering}[=theorem \ref{main6} and corollary \ref{cor6}] Let $X=S^{[3]}$, where $S$ is a general $K3$ surface of genus $9$. Let $A\subset X$ be a general fibre of the Lagrangian
 fibration $\phi$. Then $A\in A^3_{(0)}(X)$ and
   \[  A^2_{hom}(X)\ \xrightarrow{\cdot A}\ A^5(X)\ \xrightarrow{\cdot D}\ A^6(X) \]
   is the zero map, for any divisor $D\in A^1(X)$.
   
   (In particular, let $b\in A^6(X)$ be a $0$--cycle of the form $b=A\cdot D\cdot c\,$, where $D\in A^1(X)$ and $c\in A^2(X)$. Then $b$ is rationally trivial if and only if $b$ has degree $0$.)
  \end{nonumbering}

 Theorems \ref{main4} and \ref{main6} are deduced from a more general statement (theorem \ref{main}). Roughly speaking, this general statement says that if a subvariety $Z$ of $X$ exists relatively (i.e. there exists a subvariety $\Zz$ in the family $\XX\to B$ of all Hilbert schemes of $K3$ surfaces of fixed genus $g\le 10$, such that $Z$ is the restriction of $\Zz$ to the fibre $X$), then the behaviour of $Z$ in the cohomology ring of $X$ can be translated into consequences about the behaviour of $Z$ in the Chow ring of $X$. This type of statement, highlighting 
  the distinguished behaviour of cycles that exist relatively, is a typical feature of the technique of ``spread'' of algebraic cycles as developed by Voisin \cite{V0}, \cite{V1}, \cite{V8}, \cite{Vo}, \cite{Vo2}, which we employ to prove theorem \ref{main}.\footnote{NB: after the present paper was written, the paper \cite{FLV} appeared, which explores closely related questions. Both the present paper and \cite{FLV} are inspired by \cite{PSY}.}
  
 One ingredient in the proof that may be of independent interest is a ``hard Lefschetz'' type of statement for certain pieces of the Chow groups of Hilbert schemes:
 
 \begin{nonumbering}[=corollary \ref{hardhilb}] Let $S$ be a $K3$ surface of genus $g\le 10$, and let $X=S^{[m]}$ be the Hilbert scheme of length $m$ subschemes of $S$. 
 There exists an ample line bundle $L$ on $X$ such that
    \[  \cdot L^{m-1}\colon\ \ \ A^2_{(2)}(X)\ \to\ A^{2m}_{(2)}(X) \]
    is an isomorphism. 
      \end{nonumbering}
  
  This is also proven using the method of ``spread''. It would be interesting to prove the results of this note for other hyperk\"ahler varieties. Unfortunately, our method runs into problems for Hilbert schemes of high genus $K3$ surfaces (this is due to the lack of Mukai models for high genus $K3$ surfaces).

 \vskip0.6cm

\begin{convention} In this article, the word {\sl variety\/} will refer to a reduced irreducible scheme of finite type over $\C$. A {\sl subvariety\/} is a (possibly reducible) reduced subscheme which is equidimensional. 

{\bf All Chow groups will be with rational coefficients}: we will always write $A_j(X)$ for the Chow group of $j$--dimensional cycles on $X$ with $\QQ$--coefficients; for $X$ smooth of dimension $n$ the notations $A_j(X)$ and $A^{n-j}(X)$ are used interchangeably. 

The notations $A^j_{hom}(X)$, $A^j_{AJ}(X)$ will be used to indicate the subgroups of homologically trivial, resp. Abel--Jacobi trivial cycles.
For a morphism $f\colon X\to Y$, we will write $\Gamma_f\in A_\ast(X\times Y)$ for the graph of $f$.
The contravariant category of Chow motives (i.e., pure motives with respect to rational equivalence as in \cite{Sc}, \cite{MNP}) will be denoted $\MM_{\rm rat}$.


We will use $H^j(X)$ 
to indicate singular cohomology $H^j(X,\QQ)$.
\end{convention}

\section{Preliminaries}

\subsection{Quotient varieties}
\label{ssquot}

\begin{definition} A {\em projective quotient variety\/} is a variety
  \[ X=Y/G\ ,\]
  where $Y$ is a smooth projective variety and $G\subset\hbox{Aut}(Y)$ is a finite group.
  \end{definition}
  
 \begin{proposition}[Fulton \cite{F}]\label{quot} Let $X$ be a projective quotient variety of dimension $n$. Let $A^\ast(X)$ denote the operational Chow cohomology ring. The natural map
   \[ A^i(X)\ \to\ A_{n-i}(X) \]
   is an isomorphism for all $i$.
   \end{proposition}
   
   \begin{proof} This is \cite[Example 17.4.10]{F}.
      \end{proof}

\begin{remark} It follows from proposition \ref{quot} that the formalism of correspondences goes through unchanged for projective quotient varieties (this is also noted in \cite[Example 16.1.13]{F}). We can thus consider motives $(X,p,0)\in\MM_{\rm rat}$, where $X$ is a projective quotient variety and $p\in A^n(X\times X)$ is a projector. For a projective quotient variety $X=Y/G$, one readily proves (using Manin's identity principle) that there is an isomorphism
  \[  h(X)\cong h(Y)^G:=(Y,\Delta^G_Y,0)\ \ \ \hbox{in}\ \MM_{\rm rat}\ ,\]
  where $\Delta^G_Y$ denotes the idempotent ${1\over \vert G\vert}{\sum_{g\in G}}\Gamma_g$.  
  \end{remark}

 \subsection{The Fourier decomposition}

 \begin{theorem}[Shen--Vial \cite{SV}]\label{S2} Let $S$ be a $K3$ surface, and let $X=S^{[2]}$ be the Hilbert scheme of length $2$ subschemes of $S$. There is a decomposition 
   \[ A^i(X) =\bigoplus_{\stackrel{0\le j\le i}{j\ {\scriptstyle even}}} A^i_{(j)}(X)\ ,\]
   with the following properties:
   

   \noindent
   (\rom1) $A^\ast_{(\ast)}(X)$ is a bigraded ring;
   
 \noindent
   (\rom2) $A^i_{(j)}(X)\subset A^i_{hom}(X)$ for $j>0$.
   
   \end{theorem}
 
 \begin{proof} This is essentially \cite[Theorem 2]{SV}, combined with the fact that there is a class $L\in A^2(X\times X)$ lifting the Beauville--Bogomolov class and satisfying certain equalities, which is \cite[Part 2]{SV}.
 \end{proof}

 \subsection{MCK decomposition}
\label{ss1}

\begin{definition}[Murre \cite{Mur}] Let $X$ be a smooth projective variety of dimension $n$. We say that $X$ has a {\em CK decomposition\/} if there exists a decomposition of the diagonal
   \[ \Delta_X= \pi^X_0+ \pi^X_1+\cdots +\pi^X_{2n}\ \ \ \hbox{in}\ A^n(X\times X)\ ,\]
  such that the $\pi^X_i$ are mutually orthogonal idempotents in $A^n(X\times X)$ and $(\pi^X_i)_\ast H^\ast(X)= H^i(X)$.
  
  (NB: ``CK decomposition'' is shorthand for ``Chow--K\"unneth decomposition''.)
\end{definition}

\begin{remark} The existence of a CK decomposition for any smooth projective variety is part of Murre's conjectures \cite{Mur}, \cite{J2}, \cite{J4}. 
\end{remark}

\begin{definition}[Shen--Vial \cite{SV}] Let $X$ be a smooth projective variety of dimension $n$. Let $\Delta_X^{sm}\in A^{2n}(X\times X\times X)$ be the class of the small diagonal
  \[ \Delta_X^{sm}:=\bigl\{ (x,x,x)\ \vert\ x\in X\bigr\}\ \subset\ X\times X\times X\ .\]
  An {\em MCK decomposition\/} is a CK decomposition $\{\pi^X_i\}$ of $X$ that is {\em multiplicative\/}, i.e. it satisfies
  \[ \pi^X_k\circ \Delta_X^{sm}\circ (\pi^X_i\times \pi^X_j)=0\ \ \ \hbox{in}\ A^{2n}(X\times X\times X)\ \ \ \hbox{for\ all\ }i+j\not=k\ .\]
  
 (NB: ``MCK decomposition'' is shorthand for ``multiplicative Chow--K\"unneth decomposition''.) 
  
 A {\em weak MCK decomposition\/} is a CK decomposition $\{\pi^X_i\}$ of $X$ that satisfies
    \[ \Bigl(\pi^X_k\circ \Delta_X^{sm}\circ (\pi^X_i\times \pi^X_j)\Bigr){}_\ast (a\times b)=0 \ \ \ \hbox{for\ all\ } a,b\in\ A^\ast(X)\ .\]
  \end{definition}
  
  \begin{remark} The small diagonal (seen as a correspondence from $X\times X$ to $X$) induces the {\em multiplication morphism\/}
    \[ \Delta_X^{sm}\colon\ \  h(X)\otimes h(X)\ \to\ h(X)\ \ \ \hbox{in}\ \MM_{\rm rat}\ .\]
 Suppose $X$ has a CK decomposition
  \[ h(X)=\bigoplus_{i=0}^{2n} h^i(X)\ \ \ \hbox{in}\ \MM_{\rm rat}\ .\]
  By definition, this decomposition is multiplicative if for any $i,j$ the composition
  \[ h^i(X)\otimes h^j(X)\ \to\ h(X)\otimes h(X)\ \xrightarrow{\Delta_X^{sm}}\ h(X)\ \ \ \hbox{in}\ \MM_{\rm rat}\]
  factors through $h^{i+j}(X)$.
  
  If $X$ has a weak MCK decomposition, then setting
    \[ A^i_{(j)}(X):= (\pi^X_{2i-j})_\ast A^i(X) \ ,\]
    one obtains a bigraded ring structure on the Chow ring: that is, the intersection product sends $A^i_{(j)}(X)\otimes A^{i^\prime}_{(j^\prime)}(X) $ to  $A^{i+i^\prime}_{(j+j^\prime)}(X)$.
    
      It is expected (but not proven !) that for any $X$ with a weak MCK decomposition, one has
    \[ A^i_{(j)}(X)\stackrel{??}{=}0\ \ \ \hbox{for}\ j<0\ ,\ \ \ A^i_{(0)}(X)\cap A^i_{hom}(X)\stackrel{??}{=}0\ ;\]
    this is related to Murre's conjectures B and D, that have been formulated for any CK decomposition \cite{Mur}.

  The property of having an MCK decomposition is severely restrictive, and is closely related to Beauville's ``(weak) splitting property'' \cite{Beau3}. For more ample discussion, and examples of varieties with an MCK decomposition, we refer to \cite[Section 8]{SV}, as well as \cite{V6}, \cite{SV2}, \cite{FTV}, \cite{LV}.
    \end{remark}
    
  

\subsection{Relative MCK for $S^{m}$ and for $S^{(m)}$}


\begin{theorem}[Vial \cite{V6}]\label{charles} Let $S$ be a projective $K3$ surface, and let $X=S^{[m]}$ be the Hilbert scheme of length $m$ subschemes of $S$. Then $X$ has a self--dual MCK decomposition $\{ \Pi^X_i\}$. In particular, $A^\ast(X)=A^\ast_{(\ast)}(X)$ is a bigraded ring, where
  \[ A^i(X)=\bigoplus_{j= 2i-2n }^i A^i_{(j)}(X)\ ,\]
  and $A^i_{(j)}(X)=0$ for $j$ odd. In case $m=2$, the bigrading $A^\ast_{(\ast)}(X)$ coincides with the one given by the Fourier decomposition of theorem \ref{S2}.
\end{theorem}

\begin{proof} This is \cite[Theorems 1 and 2]{V6}. The last statement is \cite[Theorem 15.8]{SV}, plus the fact that for $m=2$ the MCK decomposition of \cite{V6} coincides with the one of \cite{SV}.
\end{proof}

\begin{remark} Let $X$ be as in theorem \ref{charles} and suppose $m=2$ (i.e. $X=S^{[2]}$ is a hyperk\"ahler fourfold). Then the bigrading $A^\ast_{(\ast)}(X)$ of theorem \ref{charles} has an interesting alternative description in terms of a Fourier operator on Chow groups (theorem \ref{S2}). For $m>2$, there is no such ``Fourier operator'' description of the bigrading $A^\ast_{(\ast)}(S^{[m]})$; the bigrading is defined exclusively by an MCK decomposition.

Another point particular to $m=2$ is that (thanks to \cite{SV}) we know that
  \[ A^i_{(j)}(S^{[2]})=0\ \ \ \forall j<0\ .\]
  This vanishing statement is (conjecturally true but) open for $S^{[m]}$ with $m>2$.
\end{remark}

\begin{notation} Let $\Ss\to B$ be a family (i.e., a smooth projective morphism). For $r\in\NN$, we write $\Ss^{r/B}$ for the relative $r$--fold fibre product
  \[ \Ss^{r/B}:= \Ss\times_B \Ss\times_B \cdots \times_B \Ss \ \]
  ($r$ copies of $\Ss$).
  \end{notation}

\begin{proposition}\label{prod} Let $\Ss\to B$ be a family of $K3$ surfaces. There exist relative correspondences 
   \[  \Pi_j^{\Ss^{m/B}}\ \ \in A^{2m}(\Ss^{m/B}\times \Ss^{m/B})\ \ \  (j=0,2,4,\ldots, 4m)\ ,\]
  such that for each $b\in B$, the restriction
  \[ \Pi_j^{(S_b)^m} := \Pi_j^{\Ss^{m/B}}\vert_{(S_b)^{2m}}\ \ \ \in A^{2m}((S_b)^m\times (S_b)^m)\]
  defines a self--dual MCK decomposition for $(S_b)^m$.
  \end{proposition}

\begin{proof}

 On any $K3$ surface $S_b$, there is the distinguished $0$--cycle ${\mathfrak o}_{S_b}$ such that $c_2(S_b)=24 {\mathfrak o}_{S_b}$ \cite{BV}. Let $p_i\colon \Ss^{m/B}\to \Ss$, $i=1,\ldots,m$, denote the projections to the $i$th factor. Let $T_{\Ss/B}$ denote the relative tangent bundle.
The assignment
  \[ \begin{split} \Pi_0^\Ss &:= (p_1)^\ast \bigl({1\over 24} c_2(T_{\Ss/B})\bigr) \ \ \ A^2(\Ss\times_B \Ss)\ ,\\
                         \Pi_4^\Ss &:= (p_2)^\ast \bigl({1\over 24} c_2(T_{\Ss/B})\bigr) \ \ \ A^2(\Ss\times_B \Ss)\ ,\\
                         \Pi_2^\Ss &:= \Delta_\Ss - \Pi_0^\Ss - \Pi_4^\Ss\\
                    \end{split}\]
          defines (by restriction) an MCK decomposition for each fibre, i.e.
          \[  \Pi_j^{S_b}:= \Pi_j^\Ss\vert_{S_b\times S_b}\ \ \ \in A^2(S_b\times S_b)\ \ \ (j=0,2,4) \]
          is an MCK decomposition for any $b\in B$ \cite[Example 8.17]{SV}.
          
  Next, we consider the $m$--fold relative fibre product $\Ss^{m/B}$. Let
    \[ p_{i,j}\colon \Ss^{2m/B}\ \to\ \Ss^{2/B} \ \ \ (1\le i<j\le 2m)\]
    denote projection to the $i$-th and $j$-th factor. We define
    \[  \begin{split}  \Pi_j^{\Ss^{m/B}} := {\displaystyle \sum_{k_1+k_2+\cdots+k_m=j}}  (p_{1,m+1})^\ast ( \Pi_{k_1}^{\Ss})\cdot (p_{2,m+2})^\ast (\Pi_{k_2}^\Ss)\cdot\ldots\cdot   
                                                                                       (p_{m,2m})^\ast ( \Pi_{k_m}^{\Ss})&\\           \ \ \ \in A^{2m}(\Ss^{4m/B})\ ,\ \ \ 
    (j=0,2,4,\ldots,4m)\ &.\\
    \end{split}\]
    By construction, the restriction to each fibre induces an MCK decomposition (the ``product MCK decomposition'')
    \[ \begin{split} \Pi_j^{(S_b)^m} :=  \Pi_j^{\Ss^{m/B}}\vert_{(S_b)^{2m}} = {\displaystyle \sum_{k_1+k_2+\cdots+k_m=j}}  \Pi_{k_1}^{S_b}\times \Pi_{k_2}^{S_b}\times\cdots
        \times \Pi_{k_m}^{S_b}\ \ \ \in A^{2m}((S_b)^{4m})\ ,&\\
        \ \ \ (j=0,2,4,\ldots,4m)\ .&\\
        \end{split}\]
    \end{proof}

\begin{remark}\label{relsym} Let $\Ss\to B$ be a family of $K3$ surfaces. Let
  \[ \Ss^{(m)}:= \Ss^{m/B}/\Sy_m \]
  denote the associated family of $m$--fold symmetric products (here $\Sy_m$ denotes the symmetric group on $m$ factors). The construction of the $\Pi_j^{\Ss^{m/B}}$ is $\Sy_m$--invariant, and so it induces relative projectors
  \[ \Pi_j^{\Ss^{(m)}}\ \ \ \in A^{2m}(\Ss^{(m)}\times_B \Ss^{(m)})\ .\]
  \end{remark}

      \begin{proposition}\label{prod2} Let $\Ss\to B$ be a family of $K3$ surfaces. There exist relative correspondences
    \[  \Theta^\prime_1\ ,\ldots,\ \Theta^\prime_m\in A^{2m}(\Ss^{m/B}\times_B \Ss)\ ,\ \ \ \Xi^\prime_1\ ,\ldots, \ \Xi^\prime_m\in  A^{2}(\Ss\times_B \Ss^{m/B})  \]
    such that for each $b\in B$, the composition
    \[  \begin{split}   A^{2m}_{(2)}\bigl((S_b)^m\bigr)\ \xrightarrow{((\Theta^\prime_1\vert_{(S_b)^{m+1}})_\ast,\ldots, (\Theta^\prime_m\vert_{(S_b)^{m+1}})_\ast)}\
             A^2(S_b)\oplus \cdots \oplus A^2(S_b)&\\
             \ \ \ \ \ \ \xrightarrow{((\Xi^\prime_1+\ldots+\Xi^\prime_m)\vert_{(S_b)^{m+1}})_\ast}\ A^{2m}\bigl((S_b)^m\bigr)&\\
             \end{split} \]
      is the identity.
    \end{proposition}
    
    \begin{proof} 
   As before, let 
    \[ p_{i,j}\colon\ \ \  \Ss^{2m/B}\ \to\ \Ss^{2/B} \ \ \ (1\le i<j\le 2m)\]
    denote projection to the $i$-th and $j$-th factor, and let 
    \[ p_i\colon \ \ \  \Ss^{m/B}\ \to\ \Ss \ \ \  (1\le i\le m) \]
    denote projection to the $i$--th factor.
            
        We now claim that for each $b\in B$, there is equality
       \begin{equation}\label{both} \begin{split} ( \Pi_{4m-2}^{\Ss^{m/B}})\vert_{(S_b)^{2m}} =  {1\over 24^{2m-2}}\Bigl(  {}^t \Gamma_{p_{1}}\circ \Pi_2^\Ss\circ \Gamma_{p_{1}}\circ 
               \bigl(   (p_{1,m+1})^\ast (\Delta_\Ss )\cdot \prod_{\stackrel{2\le j\le 2m}{j\not=m+1}} (p_{j})^\ast c_2(T_{\Ss/B}) &\bigr)\\  
               + \ldots +
                           {}^t \Gamma_{p_{m}}\circ \Pi_2^\Ss\circ \Gamma_{p_{m}}\circ 
               \bigl(   (p_{m,2m})^\ast (\Delta_\Ss )\cdot \prod_{\stackrel{1\le j\le 2m-1}{j\not=m}}(p_{j})^\ast c_2(T_{\Ss/B})    \bigr)   \Bigr)&\vert_{(S_b)^{2m}}\\
               \ \ \ \hbox{in}\ A^{2m}((S_b)^{m}&\times (S_b)^{m})\ .\\
               \end{split}\end{equation}
         Indeed, using Lieberman's lemma \cite[16.1.1]{F}, we find that
         \[ \begin{split} ( {}^t \Gamma_{p_{1}}\circ &\Pi_2^\Ss\circ \Gamma_{p_{1}})\vert_{(S_b)^{2m}} = \bigl(({}^t \Gamma_{p_{1,m+1}})_\ast 
         (\Pi_2^{\Ss})\bigr)\vert_{(S_b)^{2m}} =
           \bigl((p_{1,m+1})^\ast (\Pi_2^{\Ss})\bigr)\vert_{(S_b)^{2m}}\ ,\\
                           &\vdots\\
           ( {}^t \Gamma_{p_{m}}\circ &\Pi_2^\Ss\circ \Gamma_{p_{m}})\vert_{(S_b)^{2m}} = \bigl(({}^t \Gamma_{p_{m,2m}})_\ast 
         (\Pi_2^{\Ss})\bigr)\vert_{(S_b)^{2m}} =
           \bigl((p_{m,2m})^\ast (\Pi_2^{\Ss})\bigr)\vert_{(S_b)^{2m}}\ .\\
           \end{split}           \]
           
       Let us now (by way of example) consider the first summand of the right--hand--side of (\ref{both}). For brevity, let
        \[ P\colon\ \ \  (S_b)^{3m}\ \to\ (S_b)^{2m} \]
        denote the projection on the first $m$ and last $m$ factors. Writing out the definition of composition of correspondences,
        we find that
       \[ \begin{split}     &{1\over 24^{2m-2}}\Bigl(  {}^t \Gamma_{p_{1}}\circ \Pi_2^\Ss\circ \Gamma_{p_{1}}\circ 
               \bigl(   (p_{1,m+1})^\ast (\Delta_\Ss )\cdot \prod_{\stackrel{2\le j\le 2m}{j\not=m+1}} (p_{j})^\ast c_2(T_{\Ss/B}) \bigr)\Bigr)\vert_{(S_b)^{2m}} =\\
                & {1\over 24^{2m-2}}\Bigl(   \bigl((p_{1,m+1})^\ast (\Pi_2^{S_b})\bigr)   \circ 
               \bigl(   (p_{1,m+1})^\ast (\Delta_{S_b} )\cdot \prod_{\stackrel{2\le j\le 2m}{j\not=m+1}} (p_{j})^\ast c_2(T_{S_b}) \bigr)\Bigr) =\\ 
               & P_\ast    \Bigl( \bigl( (\Delta_{S_b})_{(1,m+1)} \times {\mathfrak o}_{S_b} \times\cdots\times{\mathfrak o}_{S_b} \times S_b\times\cdots\times S_b\bigr)\cdot \\
               &\ \ \ \ \ \ \ \ \bigl( S_b\times\cdots\times S_b\times (\Pi_2^{S_b})_{(m+1,2m+1)}\times S_b\times\cdots\times S_b     \bigr)  \Bigr)= \\
               & P_\ast \Bigl(  \bigl((\Delta_{S_b}\times S_b)\cdot (S_b\times\Pi_2^{S_b})\bigr)_{(1,m+1,2m+1)}\times  {\mathfrak o}_{S_b}\times\cdots\times {\mathfrak o}_{S_b}\times S_b\times\cdots\times S_b\Bigr)=\\
               & \Pi_2^{S_b}\times \Pi_4^{S_b}\times\cdots\times \Pi_4^{S_b}\ \ \ \ \ \ \hbox{in}\ A^{2m}\bigl( (S_b)^m\times (S_b)^m\bigr)\ .\\
               \end{split}\]  
               (Here, we use the notation $(C)_{(i, j)}$ to indicate that the cycle $C$ lies in the $i$th and $j$th factor, and likewise for $(D)_{(i,j,k)}$.)           
                          
      Doing the same for the other summands in (\ref{both}), one convinces oneself that both sides of (\ref{both}) are equal to the fibrewise product Chow--K\"unneth component
         \[  \Pi_{4m-2}^{(S_b)^m}=\Pi_2^{S_b}\times \Pi_4^{S_b}\times  \cdots\times\Pi_4^{S_b}  +\cdots  +      \Pi_4^{S_b}\times\cdots\times\Pi_4^{S_b}\times \Pi_2^{S_b}    \ \ \ \in A^{2m}((S_b)^m\times (S_b)^m)\ ,\]
         thus proving the claim.

  Let us now define
      \[ \begin{split}
           \Theta^\prime_i&:={1\over 24^{2m-2}} \, \Gamma_{p_{i}}\circ 
               \bigl(   (p_{i,m+i})^\ast (\Delta_\Ss )\cdot \prod_{\stackrel{j\in [1,2m]}{ j\not\in\{i,m+i\}}} (p_{j})^\ast c_2(T_{\Ss/B})    \bigr)\ \ \ \in A^{2m}((\Ss^{m/B})\times_B \Ss)\ ,\\  
                 \Xi^\prime_i&:= {}^t \Gamma_{p_{i}}\circ \Pi_2^\Ss\ \ \ \ \ \ \in A^2(\Ss\times_B (\Ss^{m/B})) \ ,\\
                       \end{split}\]
                       where $1\le i\le m$.
    It follows from equation (\ref{both}) that there is equality 
      \begin{equation}\label{transp} \begin{split} \Bigl( (\Xi^\prime_1\circ \Theta^\prime_1 + \cdots +\Xi^\prime_m\circ \Theta^\prime_m)\vert_{(S_b)^{2m}}   \Bigr){}_\ast =
      \bigl(\Pi_{4m-2}^{(S_b)^m}\bigr){}_\ast\colon &\\
           \ \ A^{i}_{(j)}\bigl((S_b)^m\bigr)\ \to\ A^{i}_{(j)}\bigl((S_b)^m\bigr)&\ \ \ \forall b\in B\ \ \ \forall (i,j)\ .\\
      \end{split}\end{equation}      
      Taking $(i,j)=(2m,2)$, this proves the proposition.      
         \end{proof}

 The following is a version of proposition \ref{prod2} for the group $A^2_{(2)}((S_b)^m)$:     
  
 \begin{proposition}\label{prod3} Let $\Ss\to B$ be a family of $K3$ surfaces. There exist relative correspondences
    \[  \Theta_1\ ,\ldots,\ \Theta_m\in A^{2m}(\Ss\times_B (\Ss^{m/B}))\ ,\ \ \ \Xi_1\ ,\ldots, \ \Xi_m\in  A^{2}( (\Ss^{m/B})\times_B \Ss)  \]
    such that for each $b\in B$, the composition
    \[  \begin{split}   A^{2}_{(2)}\bigl((S_b)^m\bigr)\ \xrightarrow{((\Xi_1\vert_{(S_b)^{m+1}})_\ast,\ldots, (\Xi_m\vert_{(S_b)^{m+1}})_\ast)}\
             A^2(S_b)\oplus \cdots \oplus A^2(S_b)&\\
             \ \ \ \ \ \ \xrightarrow{((\Theta_1+\ldots+\Theta_m)\vert_{(S_b)^{m+1}})_\ast}\ A^{2}\bigl((S_b)^m\bigr)&\\
             \end{split} \]
      is the identity.
  \end{proposition} 
  
  \begin{proof} One may take
    \[ \begin{split}   \Theta_i&:= {}^t \Theta^\prime_i\ \ \ \in\ A^{2m}(\Ss\times_B (\Ss^{m/B}))\ ,\\
                             \Xi_i&:= {}^t \Xi^\prime_i   \ \ \ A^2((\Ss^{m/B})\times_B \Ss)\ \ \ \ (i=1,\ldots,m)\ .\\
                    \end{split}\]
       By construction, the product MCK decomposition $\{ \Pi_i^{(S_b)^m}\}$ satisfies
          \[ \Pi_2^{(S_b)^m} = {}^t \bigl(\Pi_{4m-2}^{(S_b)^m}\bigr)\ \ \ \hbox{in}\ A^{2m}\bigl( (S_b)^m\times (S_b)^m\bigr)\ .\]   
     Hence, the transpose of equation (\ref{transp}) gives the equality
        \[    \begin{split}    \bigl( \Pi_2^{(S_b)^m} \bigr){}_\ast = \bigl( {}^t (\Pi_{4m-2}^{(S_b)^m})\bigr){}_\ast = \bigl( {}^t \Theta^\prime_1\circ {}^t \Xi^\prime_1+\ldots+{}^t \Theta^\prime_m\circ {}^t \Xi^\prime_m\bigr){}_\ast\colon&\ \\ \  \ A^{i}_{(j)}\bigl((S_b)^m\bigr)\ \to\ A^{i}_{(j)}\bigl((S_b)^m\bigr)\ \ &\ \forall b\in B\ \ \ \forall (i,j)\ .\\
        \end{split} \]
        Taking $(i,j)=(2,2)$, this proves the proposition.
         \end{proof}

     \subsection{Spread}
     \label{sss}
 
 The following result, taken from Voisin's method of ``spread'' \cite{V0}, \cite{V1}, \cite{Vo}, \cite{Vo2}, will be an essential ingredient in this note. This result acts as a magic wand, taking a homological equivalence and transmuting it into a rational equivalence.
   
 \begin{proposition}[Voisin \cite{V0}]\label{voisin1} Let $M$ be a smooth projective variety of dimension $r+2$, and
      assume $M$ has trivial Chow groups (i.e. $A^\ast_{hom}(M)=0$). Let $L_1,\ldots,L_r$ be very ample line bundles on $M$, and let 
      \[ \YY\ \to\ B \]
      be the universal family of smooth complete intersections 
       \[ Y_b=M\cap D_1\cap\cdots\cap D_r\ ,\ \ \ D_j\in\vert L_j\vert\ .\]
       Let $R\in A^2(\YY\times_B \YY)$ be a relative correspondence such that
      \[  R\vert_{{Y_b\times Y_b}}=0\ \ \in H^{4}({Y_b\times Y_b})\ \ \ \hbox{for\ very\ general\ }b\in B\ .\]
   Then there exists $\delta\in A^2(M\times M)_{}$ such that
    \[     R\vert_{{Y_b\times Y_b}}= \delta\vert_{Y_b\times Y_b}  \ \ \in A^{2}({Y_b\times Y_b})\ \ \ \forall\ b\in B\ .\]  
        \end{proposition}

\begin{proof} This follows from the argument of \cite{V0}. More in detail: a Leray spectral sequence argument \cite[Lemmas 3.11 and 3.12]{V0} shows that (after shrinking $B$) one can find $\delta\in A^2(M\times M)$ such that
  \[ R - (\delta\times B)\vert_{\YY\times_B \YY}\ \ \ \in A^2_{hom}(\YY\times_B \YY)\ .\]
 But $\wt{\YY\times_B \YY}$ (the blow--up along the relative diagonal) is a Zariski open in a smooth projective variety with trivial Chow groups \cite[Proof of Proposition 3.13]{V0}, and so
  \[ A^2_{hom}(\YY\times_B \YY)=0\ .\]
  In particular, this forces
  \[   R - (\delta\times B)\vert_{\YY\times_B \YY}=0\ \ \ \hbox{in}\  A^2_{}(\YY\times_B \YY)\ .\]  
  Restricting to a fibre, this gives
  \[  R\vert_{{Y_b\times Y_b}}= \delta\vert_{Y_b\times Y_b}  \ \ \in A^{2}({Y_b\times Y_b})\ \ \ \hbox{for\ general\ }b\in B\ .\]  
  Finally, a Hilbert schemes argument shows that the same is actually true for {\em all\/} $b\in B$.
  
  Note that in arbitrary dimension $n$, the argument of \cite{V0} is dependent on the ``Voisin standard conjecture'' \cite[Conjecture 1.6]{V0}. However (as also noted in 
  \cite[Theorem 3.14]{V0}), the Voisin standard conjecture is satisfied for $n=2$ and so is not needed as extra assumption.

(Alternatively, one could give a quick proof of proposition \ref{voisin1} along the lines of \cite[Proposition 1.6]{V1}, at least under the extra assumption that the surfaces $Y_b$ have non--zero primitive cohomology, which is OK in all cases where we apply proposition \ref{voisin1} since we only consider $K3$ surfaces $Y_b$.)
\end{proof}

\subsection{Families of $K3$ surfaces}
\label{ssfam}

  \begin{notation}\label{mukai} Let $g\in[2,10]$. Let
   \[ \PP_g:=\begin{cases}  \PP(1^3,3) &\hbox{if}\ g=2\ ,\\
                    \PP^g(\C) &\hbox{if}\ g=3,4,5\ ,\\
                    G(2,5) &\hbox{if}\ g=6\ ,\\
                       OG(5,10)         &\hbox{if}\ g=7\ ,\\
                    G(2,6) &\hbox{if}\ g=8\ ,\\
                      LG(3,6)              &\hbox{if}\ g=9\ ,\\
                        G_2^{ad}            &\hbox{if}\ g=10\ .\\
                                                         \end{cases}\]
                 (Here $\PP(1^3,3)$ denotes a weighted projective space, and $G(r,m)$ is the Grassmannian of $r$--dimensional subspaces in an $m$--dimensional 
                 vector space. The spaces $OG(r,d)$ and $LG(r,d)$ are the orthogonal, resp. lagrangian Grassmannian. The space $G_2^{ad}$ is the adjoint variety of the exceptional group $G_2$.)
       Consider the vector bundle $U_g$ on $\PP_g$ defined as
        \[ U_g:=\begin{cases}  \OO(6) &\hbox{if}\ g=2\ ,\\
                           \OO(4) &\hbox{if}\ g=3\ ,\\
                           \OO(3)\oplus \OO(2) &\hbox{if}\ g=4\ ,\\
                           \OO(2)^{\oplus 3} &\hbox{if}\ g=5\ ,\\
                           \OO(2)\oplus \OO(1)^{\oplus 3} &\hbox{if}\ g=6\ ,\\
                               \OO(1)^{\oplus 8}    &\hbox{if}\ g=7\ ,\\               
                                               \OO(1)^{\oplus 6} &\hbox{if}\ g=8\ \\
                                   \OO(1)^{\oplus 4}            &\hbox{if}\ g=9\ ,\\
                                     \OO(1)^{\oplus 3}          &\hbox{if}\ g=10\ .\\
                           \end{cases}\]
                           (here $\OO(i)$ on a Grassmannian refers to the Pl\"ucker embedding).
                           
          Let $B_g\subset\PP H^0(\PP_g,U_g)$ denote the Zariski open parametrizing smooth sections, and let
          \[ \Ss_g  := \bigl\{ (x,s)\ \vert\ s(x)=0\bigr\}\ \ \    \subset\ \PP_g\times B_g \]
          denote the universal family.                    
           \end{notation}
 
   As shown by Mukai \cite{Muk}, a general $K3$ surface of genus $g\in[2,10]$ is isomorphic to a fibre $S_b$ of the family $\Ss_g\to B_g$ (cf. also \cite{Beau15} and \cite[Section 3.1]{IM}). 
 
 \begin{notation}\label{fam2} Let $\Ss\to B$ be one of the families $\Ss_g\to B_g$ of notation \ref{mukai}. The family of $m$--fold symmetric products is defined as
   \[ \Ss^{(m)}:= \Ss^{m/B}/\Sy_m \]
   (where $\Sy_m$ is the symmetric group on $m$ factors).
 
 The family
   \[ \Ss^{[2]}\ \to\ B \]
   is defined as follows: take $\Ss\times_{B} \Ss$ and blow--up the relative diagonal, then take the quotient for the action of $\Sy_2$ exchanging the two factors. The fibre of $\Ss^{[2]} \to B $ is the Hilbert square $(S_b)^{[2]}$ of the $K3$ surface $S_b$.
   
   Likewise, the family $\Ss^{[3]}\to B$ of Hilbert cubes can be constructed from $\Ss^{3/B}$ by blowing up various partial diagonals, and quotienting for the action of $\Sy_3$.
 \end{notation}

 \subsection{Lagrangian fibrations}
\label{ssfib}

\begin{proposition}[Mukai \cite{Muk0}]\label{fibr} Let $S$ be a general $K3$ surface of genus $5$, and let $X=S^{[2]}$ be the Hilbert scheme. There exists a Lagrangian fibration
  \[ \phi\colon\ \ X\ \to\ \PP^2\ .\] 
\end{proposition}

\begin{proof} The surface $S$ can be defined as the intersection of three quadrics $Q_1, Q_2, Q_3$ in $\PP^5(\C)$. Let $N\cong (\PP^2)^\vee$ be the net of quadrics spanned by $Q_1, Q_2, Q_3$.
Any length $2$ subscheme $\xi$ in $S$ determines a line $\ell_\xi$ in $\PP^5$. Quadrics in $N$ containing the line $\ell_\xi$ form a pencil $P_\xi\cong\PP^1$ inside $N$. Dually, this determines a point in $\PP^2$, and so we obtain a morphism
  \[ \begin{split}  \phi\colon\ \ X\ &\to\ \PP^2\ ,\\
                        \xi\ &\mapsto\ (P_\xi)^\vee\ .\\
                       \end{split}   \] 
                       
   (This fibration $\phi$ is also described in \cite[Section 2.1]{Saw} and \cite{BFu}.)           
\end{proof}

The following result generalizes proposition \ref{fibr} to Hilbert squares of other $K3$ surfaces:

\begin{proposition}[Hassett--Tschinkel \cite{HT}]\label{ht} Let $S$ be a general $K3$ surface of genus $g$, and let $X=S^{[2]}$ be the Hilbert scheme. Assume that $2g-2=2m^2$ for some integer $m>1$. Then $X$ admits a Lagrangian fibration
  \[ \phi\colon\ \ X\ \to\ \PP^2\ .\] 
  This fibration exists relatively, i.e. let $\XX\to B$ be the universal family of Hilbert squares of $K3$ surfaces of genus $g$ (notation \ref{fam2}), and let $\XX^0\to B^0$ denote the restriction to $K3$ surfaces of Picard number $1$. Then there exists a morphism
  \[ \phi_\XX\colon\ \  \XX^0\ \to\ \PP^2\times B^0 \]
  such that the restriction of $\phi_\XX$ to a fibre $X=X_b$ is the Lagrangian fibration $\phi$.
  \end{proposition}
  
  \begin{proof} The construction of $\phi$ is \cite[Proposition 7.1]{HT}. 
  
  It remains to see that the fibration exists relatively. This is clear for $m=2$ from the explicit description of $\phi$ given by proposition \ref{fibr}. For $m>2$, it follows from the deformation theoretic argument proving \cite[Proposition 7.1]{HT}.   
  
  \end{proof}
  
\begin{proposition}[Iliev--Ranestad \cite{IR}]\label{ir} Let $S$ be a general $K3$ surface of genus $9$, and let $X=S^{[3]}$ be the Hilbert cube. Then $X$ admits a Lagrangian fibration
  \[ \phi\colon\ \ X\ \to\ \PP^3\ .\] 
  This fibration exists relatively, i.e. let $\XX\to B$ be the universal family of Hilbert cubes of $K3$ surfaces of genus $9$. Then there exists an
  almost holomorphic fibration
 \[ \phi_\XX\colon\ \ \ \XX\ \dashrightarrow\ \PP^3\times B\ ,\]
 such that the restriction of $\phi_\XX$ to a general fibre $X=X_b$ is the Lagrangian fibration $\phi$.
      \end{proposition}

 \begin{proof} The construction is inspired by Mukai's construction (proposition \ref{fibr}); the ambient space $\PP^5$ in Mukai's construction is replaced by the lagrangian Grassmannian $LG(3,6)$, and lines in Mukai's construction are replaced by twisted cubic curves. Let $N\cong \PP^3$ be the space of genus $9$ prime Fano threefolds $Y_h$
 in $LG(3,6)$ containing $S$. As shown in \cite{IR}, a general length $3$ subscheme $\xi$ in $S$ determines a unique twisted cubic curve $C_\xi$ in $LG(3,6)$. There is a unique element $Y_h$ in $N$ containing $C_\xi$; this determines the morphism $\phi\colon X\to\PP^3$.
 
 As for the second assertion, this follows from the fact that this construction can be done relatively over $B$. The upshot is a rational map
  \[ \phi_\XX\colon\ \ \ \XX\ \dashrightarrow\ \PP^3\times B\ .\] 
  Since it is shown in loc. cit. that the restriction of $\phi_\XX$ to a general fibre is a morphism, the map $\phi_\XX$ is almost holomorphic.
  
  \end{proof}

\begin{remark} Generalizations of proposition \ref{ir} to Hilbert schemes $S^{[r]}$, for certain other values of the genus of $S$ and of $r$, are given in \cite{Saw} and \cite{Mar}.
\end{remark}

 \section{An intermediate result}

In this section, we prove a hard Lefschetz result for the Chow groups of $S^m$ (theorem \ref{hard}) and $S^{[m]}$ (corollary \ref{hardhilb}), where $S$ is a low genus $K3$ surface.
This will be an ingredient in the proof of the main result (theorem \ref{main}) in the next section.

 \begin{theorem}\label{hard} Let $\Ss\to B$ be the universal family of $K3$ surfaces of genus $g$, where $2\le g\le 10$ (cf. subsection \ref{ssfam}). Let $L\in A^1(\Ss^{m/B})$ be a line bundle such that the restriction $L_b$ (to the fibre over $b\in B$) is big for very general $b\in B$.
   Then
     \[  \cdot (L_b)^{2m-2}\colon\ \ \ A^2_{(2)}((S_b)^m)\ \to\ A^{2m}_{(2)}((S_b)^m) \]
     is an isomorphism for all $b\in B$.
     
     Moreover, there exists $C_b\in A^2((S_b)^m\times(S_b)^m)$ inducing the inverse isomorphism. 
       \end{theorem}
      
   \begin{proof} This is proven using the technique of spread as developed by Voisin \cite{V0}, \cite{V1}.
   Let us write
    \[ \Gamma_{L^{2m-2}}:=  (p_1)^\ast(L^{2m-2})\cdot \Delta_{\Ss^{m/B}}\ \ \ \in A^{4m-2}\bigl( (\Ss^{m/B})\times_B (\Ss^{m/B})\bigr)\ ,\]
    where 
     \[ \Delta_{\Ss^{m/B}}\ \subset\    (\Ss^{m/B})\times_B (\Ss^{m/B})\] 
     is the relative diagonal, and
     \[ p_1\colon\ \ \  (\Ss^{m/B})\times_B (\Ss^{m/B})\ \to\ \Ss^{m/B} \]
     is projection on the first factor. The relative correspondence $ \Gamma_{L^{2m-2}}$ acts on Chow groups as multiplication by $L^{2m-2}$.
     
     As ``input'', we will make use of the following result:
     
     \begin{proposition}[L. Fu \cite{LFu0}]\label{input} Let $X$ be a smooth projective variety of dimension $n$ verifying the Lefschetz standard conjecture $B(X)$. Let $L\in A^1(X)$
      be a big line bundle. Then
        \[ \cup L^{n-2}\colon\ \ \ H^2(X)/N^1 H^2(X)\ \to\ H^{2n-2}(X)/N^{n-1}H^{2n-2}(X) \]
        is an isomorphism. (Here $N^\ast$ denotes the coniveau filtration \cite{BO}, so $N^i H^{2i}(X)$ is the image of the cycle class map.)
        Moreover, there is a correspondence $C\in A^2(X\times X)$ inducing the inverse isomorphism.
        \end{proposition}
        
        \begin{proof} The first statement is (a special case of) \cite[Theorem 4.11]{LFu0}, and the second statement
        follows from the proof of \cite[Theorem 4.11]{LFu0}. Alternatively, for the second statement one could reason as follows: it follows from \cite[Lemma 3.3]{LFu0} that
          \[  \cup L^{n-2}\colon\ \ \ H^2(X)/N^1 H^2(X)\ \to\ H^{2n-2}(X)/N^{n-1}H^{2n-2}(X) \]
        is an isomorphism. Since the category of motives for numerical equivalence $\MM_{\rm num}$ is semisimple \cite{J0}, it follows that there is an isomorphism of motives
        \[   h^2(X)\oplus \bigoplus_i \LLL(m_i)\ \cong\  h^{2n-2}(X)(n-2)\oplus \bigoplus_j \LLL(m_j)\ \ \ \hbox{in}\ \MM_{\rm num}\ ,\]
        where the arrow from $h^2(X)$ to $h^{2n-2}(X)(n-2)$ is given by $\Gamma_{L^{n-2}}\in A^{2n-2}(X\times X)$, and $\LLL$ denotes the Lefschetz motive.
        Since homological and numerical equivalence coincide for $X$ and for $\LLL$, this implies there is also an isomorphism
        \[   h^2(X)\oplus \bigoplus_i \LLL(m_i)\ \cong\  h^{2n-2}(X)(n-2)\oplus \bigoplus_j \LLL(m_j)\ \ \ \hbox{in}\ \MM_{\rm hom}\ ,\]
        with the arrow from $h^2(X)$ to $h^{2n-2}(X)(n-2)$ being given by $\Gamma_{L^{n-2}}$. It follows that there exists a correspondence $C$ as required.
                     \end{proof}

     Any fibre $(S_b)^m$ of the family $\Ss^{m/B}\to B$ verifies the Lefschetz standard conjecture (the Lefschetz standard conjecture is known for products of surfaces \cite{K}). Applying proposition \ref{input}, this means that for all $b\in B$ there exists a correspondence
      \[ C_b\ \ \in A^{2}\bigl(   (S_b)^m\times     (S_b)^m\bigr) \]
     with the property that the compositions
      \[  
      H^2\bigl(  (S_b)^m\bigr)/N^1 \ \xrightarrow{\cdot    (L_b)^{2m-2}}\ H^{4m-2}\bigl( (S_b)^{m}\bigr)/N^{2m-1}\ \xrightarrow{ (C_b)_\ast} H^2\bigl( (S_b)^m\bigr)/N^1 \]
      and
       \[  
      H^{4m-2}\bigl(  (S_b)^m\bigr)/N^{2m-1} \ \xrightarrow{(C_b)_\ast}\ H^{2}\bigl( (S_b)^{m}\bigr)/N^1\ \xrightarrow{\cdot    (L_b)^{2m-2} } H^{4m-2}\bigl( (S_b)^m
      \bigr)/N^{2m-1} \]
             are the identity. In other words, for all $b\in B$ there exist 
             \[\gamma_b\ ,\ \ \  \gamma_b^\prime\in 
             A^{2m}\bigl(   (S_b)^m\times     (S_b)^m\bigr)\] 
             supported on $D_b\times D_b\subset    (S_b)^m\times     (S_b)^m$ for some divisor $D_b\subset (S_b)^m$ and such that
      \[   \begin{split}   \Pi_2^{\Ss^{m/B}}\vert_{(S_b)^m}    \circ C_b\circ \bigl( \Pi_{4m-2}^{\Ss^{m/B}}\circ \Gamma_{L^{2m-2}}\circ \Pi_2^{\Ss^{m/B}}\bigr)\vert_{(S_b)^m}  &= 
          \Pi_2^{\Ss^{m/B}}\vert_{(S_b)^m} +\gamma_b\  
      \ ,\\
              \Pi_{4m-2}^{\Ss^{m/B}}\vert_{(S_b)^m}    \circ     \bigl( \Gamma_{L^{2m-2}}       \circ  \Pi_{2}^{\Ss^{m/B}}\bigr)\vert_{(S_b)^m}\circ  C_b \circ \Pi_{4m-2}^{\Ss^{m/B}}\vert_{(S_b)^m} &= \Pi_{4m-2}^{\Ss^{m/B}}\vert_{(S_b)^m} +\gamma_b^\prime\\
                 &\ \ \ \ \ \ \hbox{in}\ H^{4m}\bigl( (S_b)^m\times (S_b)^m\bigr)\ .\\  
                       \end{split}    \]
    Applying a Hilbert schemes argument as in \cite[Proposition 3.7]{V0} (cf. also \cite[Proposition 2.10]{excubic4}), we can find a relative correspondence 
     \[ \Cc\in A^2\bigl(  (\Ss^{m/B})\times_B (\Ss^{m/B})\bigr) \]
     doing the same job as the various $C_b$, i.e. such that for all $b\in B$ one has
      \[ \begin{split}  ( \Pi_2^{\Ss^{m/B}}\circ \Cc\circ \Pi_{4m-2}^{\Ss^{m/B}}\circ \Gamma_{L^{2m-2}}\circ \Pi_2^{\Ss^{m/B}})\vert_{(S_b)^m}  &= \Pi_2^{\Ss^{m/B}}\vert_{(S_b)^m} +\gamma_b\ ,\\
            ( \Pi_{4m-2}^{\Ss^{m/B}}\circ\Gamma_{L^{2m-2}}\circ \Pi_{2}^{\Ss^{m/B}}\circ \Cc\circ \Pi_{4m-2}^{\Ss^{m/B}})\vert_{(S_b)^m}  &= \Pi_{4m-2}^{\Ss^{m/B}}\vert_{(S_b)^m}+\gamma_b^\prime\\
           &\ \ \ \ \ \  \ \hbox{in}\ H^{4m}\bigl( (S_b)^m\times (S_b)^m\bigr)\ .
            \\  \end{split}       \]   
     Applying once more the same Hilbert schemes argument \cite[Proposition 3.7]{V0}, we can also find a divisor $\DD\subset \Ss^{m/B}$ and relative correspondences
     \[ \gamma\ ,\ \ \ \gamma^\prime\ \ \ \in A^{2m}\bigl( \Ss^{m/B}\times_B \Ss^{m/B}\bigr) \]
     supported on $\DD\times_B \DD$ and doing the same job as the various $\gamma_b$, resp. $\gamma_b^\prime$. That is, $\gamma$ and $\gamma^\prime$ are such that for all $b\in B$ one has
      \[ \begin{split}  (\Pi_2^{\Ss^{m/B}}\circ \Cc\circ \Pi_{4m-2}^{\Ss^{m/B}}\circ \Gamma_{L^{2m-2}}\circ \Pi_2^{\Ss^{m/B}})\vert_{(S_b)^m}  &= (\Pi_2^{\Ss^{m/B}}+\gamma)\vert_{(S_b)^m}\ ,\\
            ( \Pi_{4m-2}^{\Ss^{m/B}}\circ \Gamma_{L^{2m-2}}\circ \Pi_{2}^{\Ss^{m/B}}\circ \Cc\circ \Pi_{4m-2}^{\Ss^{m/B}})\vert_{(S_b)^m}  &= (\Pi_{4m-2}^{\Ss^{m/B}}+\gamma^\prime)
            \vert_{(S_b)^m} \\ 
        &\ \ \ \ \ \ \     \ \hbox{in}\ H^{4m}\bigl( (S_b)^m\times (S_b)^m\bigr)\ .
            \\  \end{split}       \]                
           
      We now make an effort to rewrite this more compactly: the relative correspondences defined as
      \begin{equation}\label{defg}\begin{split} \Gamma&:=  \Pi_2^{\Ss^{m/B}}\circ \Cc\circ \Pi_{4m-2}^{\Ss^{m/B}}\circ \Gamma_{L^{2m-2}}\circ \Pi_2^{\Ss^{m/B}}  -     \Pi_2^{\Ss^{m/B}} -\gamma \ ,\\
              \Gamma^\prime&:= \Pi_{4m-2}^{\Ss^{m/B}}\circ \Gamma_{L^{2m-2}}\circ \Pi_{2}^{\Ss^{m/B}}\circ \Cc\circ \Pi_{4m-2}^{\Ss^{m/B}}  -     \Pi_{4m-2}^{\Ss^{m/B}} -\gamma^\prime\ \ \ \in A^{2m}  
              \bigl( (\Ss^{m/B})
      \times_B  (\Ss^{m/B})\bigr)  \\
      \end{split}        \end{equation}
      have the property that their restriction to any fibre is homologically trivial. That is, writing
      \[ \begin{split}  \Gamma_b &:= \Gamma\vert_{(S_b)^m\times (S_b)^m} \\
         \Gamma^\prime_b&:=  (\Gamma^\prime)\vert_{(S_b)^m\times (S_b)^m} \ \ \ \in A^{2m}_{}\bigl(  (S_b)^m\times (S_b)^m\bigr)\\
         \end{split}\]
         for the restriction to a fibre, we have that     
       \begin{equation}\label{hom0}  \Gamma_b\ ,\ \ \ \Gamma_b^\prime \ \ \ \in A^{2m}_{hom}\bigl(  (S_b)^m\times (S_b)^m\bigr)\ \ \ \forall b\in B\ .\end{equation}
       
      Let us now define the modified relative correspondences
      \[ \begin{split}   \Gamma_1 &:=\Pi_2^{\Ss^{m/B}}\circ \Gamma\circ \Pi_2^{\Ss^{m/B}}\ ,\\ 
                               \Gamma_1^\prime&:= \Pi_{4m-2}^{\Ss^{m/B}}\circ \Gamma^\prime\circ  \Pi_{4m-2}^{\Ss^{m/B}}\ \ \ \in A^{2m} \bigl( \Ss^{m/B}\times_B \Ss^{m/B}\bigr)\ .\\
                              \end{split}\]
                              
        This modification does not essentially modify the fibrewise rational equivalence class: we have
        
        \begin{equation}\label{modif} \begin{split}   
                   (\Gamma_1)_b     &= \Gamma_b + (\gamma_1)_b\ ,\\
                 (\Gamma^\prime_1)_b     &= (\Gamma^\prime)_b + (\gamma^\prime_1)_b\ \ \ \hbox{in}\  A^{2m}\bigl( (S_b)^m\times (S_b)^m\bigr)\ ,\\
                 \end{split}\end{equation}
                 where $\gamma_1, \gamma_1^\prime\in A^{2m}\bigl(\Ss^{m/B}\times_B \Ss^{m/B}\bigr)$ are relative correspondences supported on $\DD\times_B \DD$.
                 (Indeed, this is true because $(\Pi_i^{(S_b)^m})^{\circ 2}=\Pi_i^{(S_b)^m}$ for all $i$, and the relative correspondences 
                 \[\Pi_2^{\Ss^{m/B}}\circ \gamma\circ \Pi_2^{\Ss^{m/B}}\ , \ \ \Pi_{4m-2}^{\Ss^{m/B}}\circ \gamma^\prime\circ \Pi_{4m-2}^{\Ss^{m/B}}\] 
                 are still supported on $\DD\times_B \DD$.)       
                 
     As $\Gamma$ and $\Gamma^\prime$ were fibrewise homologically trivial (equation (\ref{hom0})), the same is true for $\Gamma_1$ and $\Gamma_1^\prime$:
     
     \begin{equation}\label{hom1}  (\Gamma_1)_b\ ,\ \ \ (\Gamma_1^\prime)_b \ \ \ \in A^{2m}_{hom}\bigl(  (S_b)^m\times (S_b)^m\bigr)\ \ \ \forall b\in B\ ,\end{equation}
       
We now proceed to upgrade (\ref{hom1}) to a statement concerning the action on Chow groups:

\begin{claim}\label{rat1} We have
  \[ \begin{split}  \bigl((\Gamma_1)_b\bigr){}_\ast &=0 \colon\ \ \  A^i_{hom}\bigl((S_b)^m\bigr)  \ \to\  A^i_{hom}\bigl((S_b)^m\bigr)\ \ \ \forall b\in B\ ,\\
                            \bigl((\Gamma^\prime_1)_b\bigr){}_\ast  &=0 \colon \ \ \  A^i_{hom}\bigl((S_b)^m\bigr)  \ \to\  A^i_{hom}\bigl((S_b)^m\bigr)\ \ \ \forall b\in B\ .\\
                            \end{split}\]
                            \end{claim}
     
Let us prove claim \ref{rat1} for $\Gamma_1$ (the argument for $\Gamma_1^\prime$ is only notationally different). Using proposition \ref{prod3}, one finds there is a fibrewise equality modulo rational equivalence
  \begin{equation}\label{same} (\Gamma_1)_b = \Bigl( (\sum_{i=1}^m \Xi_i\circ \Theta_i)\circ \Gamma\circ (\sum_{i=1}^m  \Xi_i\circ \Theta_i ) \Bigr){}_b\ \ \ \hbox{in}\     A^{2m}\bigl( (S_b)^m\times (S_b)^m\bigr)\ \ \ \forall b\in B\ .\end{equation}
  To rewrite this, let us define relative correspondences
  \[ \Gamma_{k,\ell}:= \Theta_k\circ \Gamma\circ \Xi_\ell\ \ \ \in A^{2}\bigl( \Ss^{}\times_B \Ss^{}\bigr)\ \ \ ( 1\le k,\ell\le m)\ .\]
  With this notation, equality (\ref{same}) becomes the equality
  \begin{equation}\label{same2} (\Gamma_1)_b = \Bigl(  \sum_{k=1}^m \sum_{\ell=1}^m \Xi_k\circ \Gamma_{k,\ell}\circ \Theta_\ell\Bigr){}_b    \ \ \ \hbox{in}\     A^{2m}\bigl( (S_b)^m\times (S_b)^m\bigr)\ \ \ \forall b\in B\ .\end{equation}
  
  As $\Gamma$ is fibrewise homologically trivial (equation (\ref{hom0})), the same is true for the various $\Gamma_{k,\ell}$:
  \[  (\Gamma_{k,\ell})_b \ \ \ \in A^{2}_{hom}(  S_b\times S_b)\ \ \ \forall b\in B\  \ \   ( 1\le k,\ell\le m)\ .\]  
 This means that we can apply Voisin's key result, proposition \ref{voisin1}, to the relative correspondence $\Gamma_{k,\ell}$. The conclusion is that for each $1\le k,\ell\le m$, there exists a cycle $\delta_{k,\ell}\in A^2(\PP\times \PP)$ (where $\PP=\PP_g$ is the homogeneous variety as in subsection \ref{ssfam}) such that
   \[  (\Gamma_{k,\ell})_b  +(\delta_{k,\ell})_b=0 \ \ \ \hbox{in}\ A^{2}_{}(  S_b\times S_b)\ \ \ \forall b\in B\  \ \ .\]    
   Since $\PP$ has trivial Chow groups, this implies in particular that
   \[  \bigl(  (\Gamma_{k,\ell})_b \bigr){}_\ast =0\colon\ \ \ A^i_{hom}(S_b)\ \to\ A^i_{hom}(S_b)\ \ \ \forall b\in B\ .\]
   In view of equality (\ref{same2}), this implies
  \[  \bigl((\Gamma_1)_b\bigr){}_\ast=0 \colon\ \ \   A^i_{hom}\bigl((S_b)^m\bigr)  \ \to\  A^i_{hom}\bigl((S_b)^m\bigr)\ \ \ \forall b\in B\   ,\]    
  as claimed.
  
  (The argument for $\Gamma_1^\prime$ is the same; it suffices to replace the use of proposition \ref{prod3} by proposition \ref{prod2}.)
  Claim \ref{rat1} is now proven.
  
  It is now time to wrap up the proof of theorem \ref{hard}. For $b\in B$ general, the restrictions $(\gamma_1)_b, (\gamma_1^\prime)_b$ of equation (\ref{modif}) will be supported on $D_b\times D_b\subset (S_b)^m\times (S_b)^m$, where $D_b\subset (S_b)^m$ is a divisor. As such, the action
    \[ \begin{split}   \bigl( (\gamma_1)_b\bigr){}_\ast\colon\ \ \ &R\bigl( (S_b)^m\bigr)\ \to\ R\bigl( (S_b)^m\bigr)\ ,\\       
                           \bigl( (\gamma^\prime_1)_b\bigr){}_\ast\colon\ \ \ &R\bigl( (S_b)^m\bigr)\ \to\ R\bigl( (S_b)^m\bigr)\ ,\\     
                       \end{split}\]
      is $0$ for general $b\in B$, where $R$ is either $A^2_{hom}$ or $A^{2m}$. Combining this observation with equation (\ref{modif}) and claim (\ref{rat1}), we find that
      \[  \begin{split}  (\Gamma_b )_\ast=0\colon\ \ \ R\bigl( (S_b)^m\bigr)\ \to\ R\bigl( (S_b)^m\bigr)\ ,\\
                               (\Gamma_b^\prime)_\ast=0\colon\ \ \ R\bigl( (S_b)^m\bigr)\ \to\ R\bigl( (S_b)^m\bigr)\ \\
                           \end{split}\]
                   (where, once more, $R$ is either $A^2_{hom}$ or $A^{2m}$).
                   
  In view of the definition (\ref{defg}) of $\Gamma, \Gamma^\prime$ (and using that the cycles $\gamma_b, \gamma^\prime_b$ occuring in (\ref{defg}) are supported in codimension $1$ for $b\in B$ general, and so act trivially on $A^2_{hom}()$ and on $A^{2m}()$), it follows that                
         \begin{equation}\label{this} \begin{split} \Bigl( \Pi_2^{(S_b)^m}\circ \Cc_b\circ \Pi_{4m-2}^{(S_b)^m}\circ (\Gamma_{L^{2m-2}})_b \circ \Pi_2^{(S_b)^m} -  \Pi_2^{(S_b)^m}\Bigr){}_\ast&=0\colon\ \ \ A^2_{hom}\bigl( (S_b)^m\bigr) \to      A^2_{hom}\bigl( (S_b)^m\bigr)\ ,\\
                        \Bigl( \Pi_{4m-2}^{(S_b)^m}\circ (\Gamma_{L^{2m-2}})_b\circ \Pi_2^{(S_b)^m}\circ \Cc_b\circ \Pi_{4m-2}^{(S_b)^m} - \Pi_{4m-2}^{(S_b)^m}\Bigr){}_\ast&=0\colon\ \ \ A^{2m}\bigl( (S_b)^m\bigr) \to A^{2m}\bigl( (S_b)^m\bigr)\ ,\\
                       \end{split}\end{equation}
                  for general $b\in B$. Since $\Pi_2^{(S_b)^m}$ acts as the identity on $A^2_{(2)}((S_b)^m)$, it follows from the first line of (\ref{this}) that
              \[       \Bigl( \Pi_2^{(S_b)^m}\circ \Cc_b\circ \Pi_{4m-2}^{(S_b)^m}\circ (\Gamma_{L^{2m-2}})_b\Bigr){}_\ast   =\ide\colon\ \ \ A^2_{(2)}\bigl( (S_b)^m\bigr)\ 
                  \to\ A^2_{(2)}\bigl( (S_b)^m\bigr)\ ;\]
             in particular
             \[ \cdot L^{2m-2}\colon\ \ \ A^2_{(2)}\bigl( (S_b)^m\bigr)\ \to\ A^{2m}_{(2)}\bigl( (S_b)^m\bigr) \]
             is injective for general $b\in B$. Likewise, it follows from the second line of (\ref{this}) that
 \[ \Bigl( \Pi_{4m-2}^{(S_b)^m}\circ (\Gamma_{L^{2m-2}})_b\circ \Pi_2^{(S_b)^m}\circ \Cc_b\Bigr){}_\ast =\ide\colon\ \ \ A^{2m}_{(2)}\bigl( (S_b)^m\bigr)\ \to\ A^{2m}_{(2)}\bigl( (S_b)^m\bigr)\]
  for general $b\in B$. However, the image of
    \[ A^2_{(2)}\bigl( (S_b)^m\bigr)\ \xrightarrow{\cdot L^{2m-2}}\ A^{2m}\bigl( (S_b)^m\bigr) \]
    is contained in $A^{2m}_{(2)}((S_b)^m)$, since $L\in A^1((S_b)^m)=A^1_{(0)}((S_b)^m)$, and so this further simplifies to
    \[ \Bigl( (\Gamma_{L^{2m-2}})_b\circ \Pi_2^{(S_b)^m}\circ \Cc_b\Bigr){}_\ast =\ide\colon\ \ \ A^{2m}_{(2)}\bigl( (S_b)^m\bigr)\ \to\ A^{2m}_{(2)}\bigl( (S_b)^m\bigr)\ \]
    for general $b\in B$. In particular,
      \[ \cdot L^{2m-2}\colon\ \ \ A^2_{(2)}\bigl( (S_b)^m\bigr)\ \to\ A^{2m}_{(2)}\bigl( (S_b)^m\bigr) \]
             is surjective for general $b\in B$.  
    
    Theorem \ref{hard} is now proven for general $b\in B$. 
    To prove the theorem for {\em all\/} $b\in B$, one observes that the above argument can be made to work ``locally around a given $b_0\in B$'', i.e. given $b_0\in B$ one can find relative correspondences $\gamma, \gamma^\prime, \ldots$ supported in codimension $1$ and in general position with respect to the fibre over $b_0$.     
                               
      \end{proof}

   Theorem \ref{hard} can be reformulated in terms of Hilbert schemes:       
 
 \begin{corollary}\label{hardhilb} Let $S_b$ be a $K3$ surface of genus $g\le 10$, and let $X=(S_b)^{[m]}$ be the Hilbert scheme of length $m$ subschemes of $S$. Let $L\in A^1(\Ss^{m/B})$ be a relatively big line bundle, and set
   \[ L_X := (f_b)^\ast (p_b)_\ast (L_b)\ \ \ \in A^1(X)\ ,\]
   where $p_b\colon (S^b)^m\to (S_b)^{(m)}$ denotes the projection, and $f_b\colon (S^b)^{[m]}\to (S_b)^{(m)}$ denotes the Hilbert--Chow morphism. Then
   \[  \cdot (L_X)^{m-1}\colon\ \ \ A^2_{(2)}(X)\ \to\ A^{2m}_{(2)}(X) \]
    is an isomorphism. 
    
  Moreover, there exists a correspondence $C\in A^2(X\times X)$ inducing the inverse isomorphism.
  \end{corollary}  
    
  \begin{proof} Let the symmetric group $\Sy_m$ act on $\Ss^{m/B}$ by permuting the factors, and let
   \[  p\colon\ \ \ \Ss^{m/B}\ \to\  \Ss^{(m)}:= \Ss^{m/B}/\Sy_m\]
   denote the quotient morphism. Theorem \ref{hard} applies to the line bundle
   \[ L^\prime:= p^\ast p_\ast (L) = \sum_{\sigma\in\Sy} \sigma^\ast(L)\ \ \ \in A^1(\Ss^{m/B})\ .\]
   There is a commutative diagram
     \[ \begin{array}[c]{ccc}
              A^2_{(2)}((S_b)^m)^{\Sy_m} & \xrightarrow{\cdot (L_b^\prime)^{m-1}}& A^{2m}_{(2)}((S_b)^m)^{\Sy_m}\\
              {\scriptstyle (p_b)^\ast} \uparrow{\scriptstyle \cong}  &&      {\scriptstyle (p_b)^\ast} \uparrow{\scriptstyle \cong} \\
               A^2_{(2)}((S_b)^{(m)}) & \xrightarrow{\cdot ((p_b)_\ast(L_b) )^{m-1}}& A^{2m}_{(2)}((S_b)^{(m)})\\   
             \end{array}\]   
       In view of theorem \ref{hard} (applied to $L^\prime$), the lower horizontal arrow is an isomorphism.      
                                   
  It follows from the de Cataldo--Migliorini isomorphism of motives \cite{CM} that there is an isomorphism (induced by a correspondence)
    \[ A^2(X)\cong A^2((S_b)^{(m)})\oplus A^1()\oplus A^0()\ ,\]
    and so in particular an isomorphism
    \[ A^2_{AJ}(X)\cong A^2_{AJ}((S_b)^{(m)})  \ .\]
    Since $A^2_{(2)}()\subset A^2_{AJ}()$, and the de Cataldo--Migliorini isomorphism respects the bigrading (by construction, the MCK decomposition for $X$ is induced by one for $(S_b)^m$), this implies that
     \[ f^\ast\colon\ \ \ A^2_{(2)}((S_b)^{(m)})\ \to\ A^2_{(2)}(X) \]
     is an isomorphism.
     
     Similarly, there is an isomorphism
     \[ f^\ast\colon\ \ \ A^{2m}((S_b)^{(m)})\ \xrightarrow{\cong}\ A^{2m}(X) \]
     which respects the bigrading. 
     
     Corollary \ref{hardhilb} now follows from what we have said above, in view of the commutative diagram 
     \[ \begin{array}[c]{ccc}
              A^2_{(2)}(X) & \xrightarrow{\cdot (L_X)^{m-1}}& A^{2m}_{(2)}(X)\\
              {\scriptstyle (f_b)^\ast} \uparrow{\scriptstyle \cong}  &&      {\scriptstyle (f_b)^\ast} \uparrow{\scriptstyle \cong} \\
               A^2_{(2)}((S_b)^{(m)}) & \xrightarrow{\cdot (p_\ast(L_b) )^{m-1}}& A^{2m}_{(2)}((S_b)^{(m)})\\   
             \end{array}\]                 
  \end{proof}

 \begin{remark}\label{question} Looking at corollary \ref{hardhilb}, one might hope that a similar result is true more generally. 
 Let $X$ be any hyperk\"ahler variety of dimension $2m$, and suppose the Chow ring of $X$ has a bigraded ring structure $A^\ast_{(\ast)}(X)$. One can ask the following questions:
 
 \noindent
 (\rom1) Let $L\in A^1(X)$ be an ample line bundle. Is it true that there are isomorphisms
   \[  \cdot L^{2m-2i+j}\colon\ \  A^i_{(j)}(X)\ \xrightarrow{\cong}\ A^{2m-i+j}_{(j)}(X)\ \ \ \hbox{for\ all\ } 0\le 2i-j\le 2m\  \ ?\]   
   
 \noindent
 (\rom2) Let $L\in A^1(X)$ be a big line bundle. Is it true that there are isomorphisms
  \[    \cdot L^{2m-i}\colon\ \  A^i_{(i)}(X)\ \xrightarrow{\cong}\ A^{2m}_{(i)}(X)\ \ \ \hbox{for\ all\ } 0\le i\le 2m\ \ ?\]     
  
  The answer to the first question is ``yes'' for generalized Kummer varieties \cite{hard}.
  The answer to both questions is ``I don't know, except for $i=2$ and $g$ low'' for Hilbert schemes of genus $g$ $K3$ surfaces. 
  
  (The question for $A^i_{ (j)}(S^{[m]})$ with $i>2$ and $g$ low becomes more complicated, as one would need an analogon of proposition \ref{voisin1} for higher fibre products 
  $\Ss^{m/B}$ with $m>2$.)
   \end{remark}

  \begin{remark} Let $X$ be either $S^m$ or $S^{[m]}$ where $S$ is a $K3$ surface of genus $g\le 10$. Let $L\in A^1(X)$ be a line bundle as in theorem \ref{hard} (resp. as in corollary \ref{hardhilb}).
 Provided $L$ is sufficiently ample, there exists a smooth complete intersection surface $Y\subset X$ defined by the linear system $\vert L\vert$. Theorem \ref{hard} (resp. corollary \ref{hardhilb}) then implies that $A^{2m}_{(2)}(X)$ is supported on $Y$, and that
   \[  A^2_{(2)}(X)\ \to\ A^2(Y) \]
   is injective. This injectivity statement is in agreement with Hartshorne's ``weak Lefschetz'' conjecture for Chow groups \cite{Ha} (we recall that it is expected that
   $A^2_{(2)}(X)=A^2_{hom}(X)$ for these $X$).
    \end{remark}

 \section{Main result}
 
 This section proves the main result of this note, theorem \ref{main}. The proof is based on the method of ``spread'' of cycles in nice families, as developed by Voisin \cite{V0}, 
 \cite{V1}, \cite{V8}, \cite{Vo}, \cite{Vo2}. The results announced in the introduction (theorems \ref{main4} and \ref{main6}) are immediate corollaries of theorem \ref{main}.

  \begin{theorem}\label{main} Let $g\in[ 2,10]$. Let $\XX=\Ss^{(m)}\to B$ denote the universal family of $m$--fold symmetric products of genus $g$ $K3$ surfaces (notation \ref{fam2}).
 Let $\Gamma\subset\XX$ be a codimension $2m-2$ subvariety, and let $\Gamma_b$ denote the restriction
  \[  \Gamma_b:=\Gamma\vert_{X_b}\ \ \ \in A^m(X_b)\ . \]
  Assume that
   \[  \cup \Gamma_b=0\colon\ \ \  H^{2,0}(X_b)\ \to\  H^{2m,2m-2}(X_b)\ \] 
    for very general $b\in B$. Then
   \[  A^2_{(2)}(X_b)\ \xrightarrow{\cdot \Gamma_b}\ A^{2m}(X_b)\ \to\ A^{2m}_{(2)}(X_b)  \]
   is the zero map, for all $b\in B$. (Here the last arrow is projection to the summand $A^{2m}_{(2)}(X_b)$.)
      \end{theorem}
   
   \begin{proof} 
      Let $f\colon\wt{\Gamma}\to\Gamma$ be a resolution of singularities, and let $\tau\colon \wt{\Gamma}\hookrightarrow \XX$ denote the composition of $f$ with the inclusion morphism $\Gamma\hookrightarrow\XX$. Let $p\colon \Ss^{m/B}\to \XX$ denote the quotient morphism.
  Let us now consider the relative correspondence
    \[ \Gamma_0:=   \Cc\circ {}^t \Gamma_p\circ  \Pi^\XX_{2m-2}\circ \Gamma_\tau\circ{}^t \Gamma_\tau \circ \Pi_2^\XX  \circ \Gamma_p\ \ \ \in A^{2m}(\Ss^{m/B}\times_B \Ss^{m/B})\ ,\]
  where $\Pi_j^\XX$ is as in remark \ref{relsym}, and $\Cc\in A^2(\Ss^{m/B}\times_B \Ss^{m/B})$ is as in the proof of theorem \ref{hard}.  
    
  By construction, for any $b\in B$, the restriction 
    \[   \Gamma_0\vert_{(S_b)^m\times (S_b)^m}\ \ \ \in A^{2m}((S_b)^m\times (S_b)^m)\]
     acts on Chow groups as
  \[     \begin{split}    ( \Gamma_0\vert_{(S_b)^m\times (S_b)^m})_\ast\colon\ \    A^2_{}((S_b)^m) \ \xrightarrow{p_\ast}\  A^2(X_b)\ \xrightarrow{(\Pi_2^{X_b})_\ast}\ A^2_{(2)}(X_b)  \ \xrightarrow{ \cdot \Gamma_b}\ A^{2m}_{}(X_b)&\\ \xrightarrow{(\Pi^{X_b}_{2m-2})_\ast}\ A^{2m}_{(2)}(X_b) 
        \ \xrightarrow{p^\ast}\ A^{2m}_{(2)}((S_b)^m)  \ \xrightarrow{(\Cc_b)_\ast}\ & A^2_{(2)}((S_b)^m)\ .\\
        \end{split}\]
  We now make the following claim: to prove theorem \ref{main} it suffices to prove that
   \begin{equation}\label{need}   \bigl((\Pi_2^{\Ss^{m/B}}\circ\Gamma_0)\vert_{(S_b)^m\times (S_b)^m}\bigr){}_\ast\stackrel{??}{=}0\colon\ \ \ A^{2}_{(2)}((S_b)^m)\ 
   \to\ A^{2}_{(2)}((S_b)^m)\ \ \  \forall b\in B\ .\end{equation}
   
  To prove the claim, we first remark that (as noted above)
    \[ \bigl(( \Gamma_0)\vert_{(S_b)^m\times (S_b)^m}\bigr){}_\ast  A^2((S_b)^m)\ \ \subset\ A^2_{(2)}((S_b)^m)\ ,\]
    and so adding $(\Pi_2^{(S_b)^m})_\ast$ doesn't change anything, i.e. the truth of statement (\ref{need}) implies that
    \[ \bigl((\Gamma_0)\vert_{(S_b)^m\times (S_b)^m}\bigr){}_\ast\stackrel{}{=}0\colon\ \ \ A^{2}_{(2)}((S_b)^m)\ 
   \to\ A^{2}_{(2)}((S_b)^m)\ \ \  \forall b\in B\ .   \]
   Next, we know from theorem \ref{hard} that $ A^{2m}_{(2)}((S_b)^m)  \ \xrightarrow{(\Cc_b)_\ast}\ A^2_{(2)}((S_b)^m)  $ is an isomorphism, and so this implies that also
   \[    \begin{split}   A^2_{}((S_b)^m) \ \xrightarrow{p_\ast}\  A^2(X_b)\ \xrightarrow{(\Pi_2^{X_b})_\ast}\ A^2_{(2)}(X_b)  \ \xrightarrow{ \cdot \Gamma_b}\ A^{2m}_{}(X_b)\ \xrightarrow{(\Pi^{X_b}_{2m-2})_\ast}\ A^{2m}_{(2)}(X_b)& \\ 
        \ \xrightarrow{p^\ast}\ A^{2m}_{(2)}((S_b)^m)& \\
        \end{split}  \] 
        is the zero map, for all $b\in B$.   
   Composing some more on both sides, this implies that also
    \[    \begin{split}   A^2(X_b)\ \xrightarrow{p^\ast} A^2_{}((S_b)^m) \ \xrightarrow{p_\ast}\  A^2(X_b)\ \xrightarrow{(\Pi_2^{X_b})_\ast}\ A^2_{(2)}(X_b)  \ \xrightarrow{ \cdot \Gamma_b}\ A^{2m}_{}(X_b)\ \xrightarrow{(\Pi^{X_b}_{2m-2})_\ast}\ A^{2m}_{(2)}(X_b)& \\ 
        \ \xrightarrow{p^\ast}\ A^{2m}_{(2)}((S_b)^m)\ \xrightarrow{p_\ast}\ A^{2m}_{(2)}(X_b)& \\
        \end{split}  \] 
        is the zero map, for all $b\in B$.   
        But $p_\ast p^\ast$ is a multiple of the identity, and so this implies that actually
 \[    A^2(X_b)\ \xrightarrow{(\Pi_2^{X_b})_\ast}\ A^2_{(2)}(X_b)  \ \xrightarrow{ \cdot \Gamma_b}\ A^{2m}_{}(X_b)\ \xrightarrow{(\Pi^{X_b}_{2m-2})_\ast}\ A^{2m}_{(2)}(X_b)  \] 
        is already the zero map, for all $b\in B$. This proves the claim, i.e. we are now reduced to proving statement (\ref{need}).

   The input we have at our disposition is that we know (from the coisotropic assumption) that
    \begin{equation}\label{coiso}  (\Gamma_0\vert_{(S_b)^m\times (S_b)^m})_\ast=0\colon\ \ \ H^{2,0}((S_b)^m)\ \to\ H^{2,0}((S_b)^m)\ \ \ \hbox{for\ very\ general\ }b\in B\ .\end{equation}
 
 We observe that (\ref{coiso}), combined with the Lefschetz (1,1) theorem, implies the following: for very general $b\in B$, there exist a curve $Y_b\subset (S_b)^m$, a divisor $D_b\subset (S_b)^m$ and a cycle $\gamma_b$ supported on $Y_b\times D_b\subset (S_b)^m\times (S_b)^m$, such that
  \[  \Gamma_0\vert_{(S_b)^m\times (S_b)^m} -\gamma_b=0\ \ \ \hbox{in}\ H^{4m}((S_b)^m\times (S_b)^m)\ .\]
     Thanks to Voisin's key result \cite[Proposition 3.7]{V0} (cf. also \cite[Proposition 4.25]{V1}), it is possible to spread out these data. That is, there exist subvarieties 
     $\YY\subset\Ss^{m/B}$, $\DD\subset\Ss^{m/B}$ of codimension $2m-1$ resp. $1$, and a cycle $\gamma\in A^{2m}(\Ss^{m/B}\times_B \Ss^{m/B})$ supported on $\YY\times_B \DD$ that does the job of the various $\gamma_b$, i.e. such that
  \[  \bigl(\Gamma_0-\gamma\bigr)\vert_{(S_b)^m\times (S_b)^m} =0\ \ \ \hbox{in}\ H^{4m}((S_b)^m\times (S_b)^m)\ \ \ \hbox{for\ very\ general\ } b\in B\ .\]  
  In other words, the relative correspondence defined as
    \[ \Gamma_1:= \Gamma_0-\gamma\ \ \ \in A^{2m}(\Ss^{m/B}\times_B \Ss^{m/B}) \] 
    has the property that 
      \begin{equation}\label{homtriv} \Gamma_1\vert_{(S_b)^m\times (S_b)^m} =0\ \ \ \hbox{in}\ H^{4m}((S_b)^m\times (S_b)^m)\ \ \ \hbox{for\ very\ general\ }b\in B\ .\end{equation}
    
It is more convenient to switch to correspondences in $A^2(\Ss\times_B \Ss)$.  
   To this end, we now define relative correspondences
   \[  \Gamma_{2}^{i,j}:=  \Xi_i\circ  \Gamma_1\circ \Theta_j\ \ \ \in   A^2(\Ss\times_B \Ss)\ \ \ (1\le i,j\le 2)\ ,  \]
   where $\Xi_i, \Theta_j$ are as in proposition \ref{prod3}. The relative correspondence $\Gamma_1$ being fibrewise homologically trivial (equation (\ref{homtriv})), the same holds 
   for the $\Gamma_2^{i,j}$:
   \[     \Gamma_2^{i,j}\vert_{S_b\times S_b} =0\ \ \ \hbox{in}\ H^{4}(S_b\times S_b)\ \ \ \hbox{for\ very\ general\ }b\in B\ \ \ (1\le i,j\le 2)\ .\]
   We can now apply proposition \ref{voisin1} to the $\Gamma_2^{i,j}$ (with $M=\PP_g$ and $L_r$ as given in subsection \ref{ssfam}). The conclusion is that there exist cycles
   $\delta^{i,j}\in A^2(\PP_g\times\PP_g)$ such that there is a fibrewise rational equivalence
     \[     \Gamma_2^{i,j}\vert_{S_b\times S_b} +\delta^{i,j}\vert_{S_b\times S_b}  =0\ \ \ \hbox{in}\ A^{2}(S_b\times S_b)\ \ \ \forall b\in B\ \ \ (1\le i,j\le 2)\ .\]     
     In particular, since ($\PP_g$ has trivial Chow groups and hence) the restriction $\delta^{i,j}\vert_{S_b\times S_b}$ acts trivially on $A^\ast_{hom}(S_b)$,
     this implies that
     \[   (\Gamma_2^{i,j}\vert_{S_b\times S_b})_\ast=0\colon\ \ \ A^\ast_{hom}(S_b)\ \to\ A^\ast(S_b)\ \ \      \forall b\in B\ \ \ (1\le i,j\le 2)\ .\] 
     This implies that also
          \[ \begin{split}    \Bigl( (\Pi_2^{\Ss^{m/B}}\circ \Gamma_1 \circ  \Pi_2^{\Ss^{m/B}})\vert_{(S_b)^m\times(S_b)^m} \Bigr){}_\ast &= \Bigl( ({\displaystyle\sum_{i,j\in\{1,2\}}} \Theta_i  \circ  \Gamma_2^{i,j}\circ \Xi_j)\vert_{S_b\times S_b} \Bigr){}{}_\ast \\  &=0\colon\ \ \   A^\ast_{hom}\bigl((S_b)^m\bigr)\ \to\ A^\ast\bigl((S_b)^m\bigr)\ \ \      \forall b\in B\ \\
          \end{split}    \]
     (here the first equality follows from proposition \ref{prod3}). 
    Since $A^2_{(2)}()\subset A^2_{hom}()$, this implies in particular that
         \[ A^2_{(2)}((S_b)^m)\ \xrightarrow{(\Gamma_1\vert_{(S_b)^m\times (S_b)^m})_\ast}\ A^2((S_b)^m)\ \xrightarrow{(\Pi_2^{(S_b)^m})_\ast}\    A^2_{(2)}((S_b)^m) \]
     is the zero map, for all $b\in B$.  
     
     For general $b\in B$, the restriction of the cycle $\delta$ to the fibre $(S_b)^m\times (S_b)^m$ will be supported on (curve)$\times$(divisor), and so will act trivially on $A^2((S_b)^m)$ for dimension reasons. That is, for general $b\in B$ we have equality
     \[ (\Gamma_1\vert_{(S_b)^m\times (S_b)^m})_\ast =     (\Gamma_0\vert_{(S_b)^m\times (S_b)^m})_\ast   \colon\ \ \ A^2((S_b)^m)\ \to\ A^2((S_b)^m)\ .\]
    The above thus implies that
      \[   A^2_{(2)}((S_b)^m)\ \xrightarrow{(\Gamma_0\vert_{(S_b)^m\times (S_b)^m})_\ast}\ A^2((S_b)^m)\ \xrightarrow{(\Pi_2^{(S_b)^m})_\ast}\    A^2_{(2)}((S_b)^m) \]
     is the zero map, for general $b\in B$. That is, we have now proven the desired statement (\ref{need}), and hence theorem \ref{main}, for general $b\in B$ (this already suffices to prove theorems \ref{main4} and \ref{main6} below). 
         
   To extend the statement to {\em all\/} $b\in B$, one notes that the construction of \cite[Proposition 3.7]{V0} (which was used above to globalize the various $\gamma_b$) can be done locally around a given $b_0\in B$.    
    \end{proof}

  As special cases of theorem \ref{main}, we can now prove the results announced in the introduction:
  
   \begin{theorem}\label{main4} Let $X=S^{[2]}$, where $S$ is a general $K3$ surface of genus $g=5$ or $g=10$. Let $A\subset X$ be a general fibre of the Lagrangian
 fibration $\phi\colon X\to\PP^2$ (subsection \ref{ssfib}). Then $A\in A^2_{(0)}(X)$ and
   \[ \cdot A\colon\ \ \ A^2_{hom}(X)\ \to\ A^4(X) \]
   is the zero map.
  \end{theorem} 
   
   \begin{proof} The first statement is easy: any point $p\in\PP^2$ is an intersection of two divisors, and so $A=\phi^{\ast}(p)\in A^2(X)$ is also an intersection of two divisors.
   
   As for the second statement, we have a decomposition
    \[  A^2_{hom}(X) = A^2_{(2)}(X)\oplus \bigl(A^2_{(0)}(X)\cap  A^2_{hom}(X) \bigr)\ \]
   (where the second summand is conjecturally zero).     
   We know that
      \[ \cdot A\colon\ \ \ A^2_{(0)}(X)\cap A^2_{hom}(X)\ \to\ A^4(X) \]
   is zero (the image lands in $A^4_{(0)}(X)\cap A^4_{hom}(X)=0$). It is thus sufficient to prove that
       \[ \cdot A\colon\ \ \ A^2_{(2)}(X)\ \to\ A^4(X) \]
   is the zero map.

   Let
     \[ h\colon X:= S^{[2]}\ \to\ S^{(2)} \]
     denote the Hilbert--Chow morphism, and let $A^\prime:=h(A)\subset S^{(2)}$. Since
     \[ h_\ast ( h^\ast(a)\cdot A) = a\cdot h_\ast(A)=a\cdot A^\prime\ \ \ \hbox{in}\ A^4(S^{(2)})\ \ \ \forall a\in A^2(S^{(2)})\ ,\]
     and 
     \[  \begin{split}    &h^\ast\colon\ \ \ A^2_{(2)}(S^{(2)})\ \to\ A^2_{(2)}(X)\ ,\\
                                &h_\ast\colon\ \ \ A^4(X)\ \to\ A^4(S^{(2)})\\
                                \end{split}\]
                                are isomorphisms, it suffices to prove that
          \[ \cdot A^\prime\colon\ \ \ A^2_{(2)}(S^{(2)})\ \to\ A^4(S^{(2)}) \]
   is the zero map. 
        
 We have seen (subsection \ref{ssfib}) that  the
      fibration $\phi$ exists relatively, and so $A\subset X$ exists relatively (i.e. there exists $\AAA\subset\XX$ such that $A$ is the restriction of $\AAA$ to the fibre $X$). It follows that $A^\prime$ also exists relatively (i.e. there exists $\AAA^\prime\subset\Ss^{(2)}$ such that $A^\prime$ is the restriction of $\AAA^\prime$ to a fibre). The result now follows from theorem \ref{main} applied to $\AAA^\prime$.
      \end{proof}
  
 \begin{corollary}\label{cor4} Let $X$ and $A$ be as in theorem \ref{main4}. Let $b\in A^4(X)$ be a $0$--cycle of the form
   $ b=A\cdot c $, where $c\in A^2(X)$. Then $b$ is rationally trivial if and only if $b$ is of degree $0$.
   \end{corollary}
   
   \begin{proof} One can decompose $c=c_0+c_2$, where $c_j\in A^2_{(j)}(X)$. Since $A\cdot c_2=0$ (theorem \ref{main4}), there is equality
     \[ b= A\cdot c_0\ \ \ \in A^4_{(0)}(X)\ .\]
     But $A^4_{(0)}(X)$ is isomorphic to $H^4(X)\cong\QQ$.
     \end{proof}

 \begin{theorem}\label{main6} Let $X=S^{[3]}$, where $S$ is a general $K3$ surface of genus $9$. Let $A\subset X$ be a general fibre of the Lagrangian
 fibration $\phi$ (subsection \ref{ssfib}). Then $A\in A^3_{(0)}(X)$ and
   \[ \cdot A\cdot D\colon\ \ \ A^2_{hom}(X)\ \to\ A^6(X) \]
   is the zero map, for any divisor $D\in A^1(X)$.
  \end{theorem}   
  
  \begin{proof} This is similar to the proof of theorem \ref{main4}. Again, the fact that $A\in A^3_{(0)}(X)$ is clear from the fact that a point in $\PP^3$ can be written as intersection of divisors.
  
  Since the fibration exists relatively the fibre $A\subset X$ exists relatively (i.e. there exists $\AAA\subset\XX$ such that $A$ is the restriction of $\AAA$ to the fibre $X$). The assumption implies that $X$ has Picard number $2$ and so the divisor $D$ also exists relatively, i.e. there exists $\DD\in A^1(\XX)$ such that $D$ is the restriction of $\DD$ to the fibre $X$. We may write $\DD$ as a sum 
    \[ \DD=\sum_j \lambda_j \DD_j\ \ \ \hbox{in}\ A^1(\XX)\ ,\]
    where $\lambda_j\in\QQ$ and $\DD_j$ is effective and in general position with respect to $\AAA$.
    
Let $h\colon \XX=\Ss^{[3]}\to \Ss^{(3)}$ denote the ``relative Hilbert--Chow morphism''.
  The result now follows upon applying theorem \ref{main} to the subvarieties
   \[  h(\AAA\cdot \DD_j)\ \ \ \subset\ \Ss^{(3)}\ .\]
    \end{proof}

\begin{corollary}\label{cor6} Let $X$ and $A$ be as in theorem \ref{main6}. Let $b\in A^6(X)$ be a $0$--cycle of the form
   $ b=A\cdot D\cdot c $, where $c\in A^2(X)$. Then $b$ is rationally trivial if and only if $b$ is of degree $0$.
   \end{corollary}

\begin{proof} One can decompose $c=c_0+c_2$, where $c_j\in A^2_{(j)}(X)$. Since $A\cdot D\cdot c_2=0$ (theorem \ref{main6}), there is equality
     \[ b= A\cdot D\cdot c_0\ \ \ \in A^6_{(0)}(X)\ .\]
     But $A^6_{(0)}(X)$ is isomorphic to $H^6(X)\cong\QQ$.
\end{proof}

   \begin{remark} Let $X_b$ be the Hilbert square of a very general $K3$ surface $S_b$. Let $\Gamma_b\subset X_b$ be a coisotropic subvariety of codimension $2$. It follows from a result of Voisin \cite[Proposition 4.2]{V14} that there is a homological equivalence
     \begin{equation}\label{cchom} \Gamma_b = \sum_i \lambda_i C_i\ \ \ \hbox{in}\ H^4(X_b)\ ,\end{equation}
    where $\lambda_i\in\QQ$ and $C_i\subset X_b$ is a constant cycle surface.   
  
  Suppose we know in addition that $\Gamma_b\in A^2_{(0)}(X_b)$ (for instance because $\Gamma_b$ is the fibre of a Lagrangian fibration). Then conjecturally, equality (\ref{cchom}) implies there is a rational equivalence
    \begin{equation}\label{ccrat} \Gamma_b \stackrel{??}{=} \sum_i \lambda_i C_i\ \ \ \hbox{in}\ A^2(X_b)\ ,\end{equation}
    which clearly would imply theorem \ref{main4}. 
    
    Unfortunately, I have not been able to prove equality (\ref{ccrat}). The approach taken here (using the Fourier decomposition of \cite{SV}) only yields the weaker statement that
    \[  \Gamma_b \stackrel{}{=} \sum_i \lambda_i C_i+ R\ \ \ \hbox{in}\ A^2(X_b)\ , \]
    where $R$ is in the ``troublesome part''  $A^2_{(0)}(X_b)\cap A^2_{hom}(X_b)$ (which is conjecturally zero). This is not sufficient to settle theorem \ref{main4}, which is why we needed to work harder to prove theorem \ref{main4}.
           \end{remark}

\vskip1cm
\begin{nonumberingt} 
 Thanks to J.S. Bach and to Yasuyo, for enduringly making the world more beautiful.
  \end{nonumberingt}

\end{document}